\documentclass[11pt]{article}
\usepackage{amssymb}
\usepackage{amsmath}
\usepackage{graphicx}
\usepackage{setspace}
\begin{document}

\newtheorem{theorem}{Theorem}[section]
\newtheorem{tha}{Theorem}
\newtheorem{conjecture}[theorem]{Conjecture}
\newtheorem{corollary}[theorem]{Corollary}
\newtheorem{lemma}[theorem]{Lemma}
\newtheorem{claim}[theorem]{Claim}
\newtheorem{proposition}[theorem]{Proposition}
\newtheorem{construction}[theorem]{Construction}
\newtheorem{definition}[theorem]{Definition}
\newtheorem{question}[theorem]{Question}
\newtheorem{problem}[theorem]{Problem}
\newtheorem{remark}[theorem]{Remark}
\newtheorem{observation}[theorem]{Observation}

\newcommand{\ex}{{\mathrm{ex}}}

\newcommand{\EX}{{\mathrm{EX}}}

\newcommand{\AR}{{\mathrm{AR}}}

\def\endproofbox{\hskip 1.3em\hfill\rule{6pt}{6pt}}
\newenvironment{proof}%
{%
\noindent{\it Proof.}
}%
{%
 \quad\hfill\endproofbox\vspace*{2ex}
}
\def\qed{\hskip 1.3em\hfill\rule{6pt}{6pt}}
\def\ce#1{\lceil #1 \rceil}
\def\fl#1{\lfloor #1 \rfloor}
\def\lr{\longrightarrow}
\def\e{\varepsilon}
\def\ex{{\rm\bf ex}}
\def\cA{{\cal A}}
\def\cB{{\cal B}}
\def\cC{{\cal C}}
    \def\cD{{\cal D}}
\def\cF{{\cal F}}
\def\cG{{\cal G}}
\def\cH{{\cal H}}
\def\ck{{\cal K}}
\def\cI{{\cal I}}
\def\cJ{{\cal J}}
\def\cL{{\cal L}}
\def\cM{{\cal M}}
\def\cP{{\cal P}}
\def\cQ{{\cal Q}}
\def\cR{{\cal R}}
\def\cS{{\cal S}}

\def\pr{{\rm Pr}}
\def\exp{{\rm  exp}}

\def\bkl{{\cal B}^{(k)}_\ell}
\def\cmkt{{\cal M}^{(k)}_{t+1}}
\def\cpkl{{\cal P}^{(k)}_\ell}
\def\cckl{{\cal C}^{(k)}_\ell}
\def\pkl{\mathbb{P}^{(k)}_\ell}
\def\ckl{\mathbb{C}^{(k)}_\ell}

\def\mC{{\cal C}}

\def\imp{\Longrightarrow}
\def\1e{\frac{1}{\e}\log \frac{1}{\e}}
\def\ne{n^{\e}}
\def\rad{ {\rm \, rad}}
\def\equ{\Longleftrightarrow}
\def\pkl{\mathbb{P}^{(k)}_\ell}

\def\c2ell{\mathbb{C}^{(2)}_\ell}
\def\codd{\mathbb{C}^{(k)}_{2t+1}}
\def\ceven{\mathbb{C}^{(k)}_{2t+2}}
\def\podd{\mathbb{P}^{(k)}_{2t+1}}
\def\peven{\mathbb{P}^{(k)}_{2t+2}}
\def\TT{{\mathbb T}}
\def\bbT{{\mathbb T}}
\def\cE{{\mathcal E}}
\def\e{\epsilon}

\def\mE{\mathbb{E}}

\def\mP{\mathbb{P}}

\def\wt{\widetilde}
\def\wh{\widehat}

\def \e{\epsilon}
\voffset=-0.5in
	
\setstretch{1.1}
\pagestyle{myheadings}
\markright{{\small \sc  Jiang, Newman:}
  {\it\small Small dense subgraphs of graphs}}

\title{\huge\bf  Small dense subgraphs of a graph}

\author{
Tao Jiang\thanks{Department of Mathematics, Miami University, Oxford,
OH 45056, USA. E-mail: jiangt@miamioh.edu. Research supported in part by
Simons Foundation Collaboration Grant \#282906 and by National Science Foundation grant DMS-1400249. 
Part of the research was conducted while the author was visiting Institut Mittag-Leffler (Djursholm, Sweden) during its special semester on graphs, hypergraphs, and computing, whose hospitality is gratefully acknowledged.} \quad \quad Andrew Newman
\thanks{Department of Mathematics, Ohio State University, Columbus, Ohio, newman.534@osu.edu.
    %{\rm\small {\jobname}.tex,}\newline ${}$ \hfill}
 \newline\indent
{\it 2010 Mathematics Subject Classifications:}
05C35.\newline\indent
{\it Key Words}:  Turan number,  Turan exponent, extremal function, cube.
} }

\date{February 10, 2015}

\maketitle
\begin{abstract}
Given a family $\cF$ of graphs, and a positive integer $n$, the Tur\'an number $ex(n,\cF)$ of $\cF$ is the maximum number of edges in an $n$-vertex graph that does not contain any member of $\cF$ as a subgraph. The order of a graph is the number of
vertices in it. In this paper, we study the Tur\'an number of the family of graphs with
bounded order and high average degree.
For every real $d\geq 2$ and positive integer $m\geq 2$, let $\cF_{d,m}$ denote the family of graphs on at most $m$ vertices
that have average degree at least $d$. It follows from the Erd\H{o}s-R\'enyi bound that $ex(n,\cF_{d,m})=\Omega(n^{2-\frac{2}{d}+\frac{c}{dm}})$, for some positive constant $c$. Verstra\"ete \cite{verstraete} asked if it is true that for each fixed $d$ there exists a function $\epsilon_d(m)$ that tends to $0$ as $m\to \infty$ such that $ex(n,\cF_{d,m})=O(n^{2-\frac{2}{d}+\epsilon_d(m)})$. We answer Verstra\"ete's question in the affirmative whenever $d$ is an integer.
We also prove an extension of the cube theorem on the Tur\'an
number of the cube $Q_3$, which partially answers a question of Pinchasi and Sharir \cite{PS}.
\end{abstract}

\section{Introduction}
We use standard notations. For undefined notations, the reader is referred to \cite{west}. In particular, the number of vertices and the number of
edges of a graph $H$ are denoted by $n(H)$ and $e(H)$, respectively.
Given a family $\cF$ of graphs, and a positive integer $n$, the Tur\'an number $ex(n,\cF)$ of $\cF$ is the maximum number of edges in an $n$-vertex graph that does not contain any member of $\cF$ as a subgraph.  When $\cF$ consists of a single graph $H$, we
write $ex(n,H)$ for $ex(n,\{H\})$. The study of Tur\'an numbers plays a central
role in extremal graph theory.  The celebrated Erd\H{o}s-Simonovits-Stone
theorem determines $ex(n,\cL)$ asymptotically for any family $\cL$ of non-bipartite graphs. However, the problem of determining $ex(n,\cL)$ when $\cL$ contains a bipartite graph is largely open with few exceptions. There are many interesting open problems concerning the Tur\'an numbers of bipartite graphs. We refer interested readers to the excellent recent survey by F\"uredi and Simonovits \cite{FS}. The following general lower bound on $ex(n,\cF)$ can be easily verified using the first moment method.

\begin{theorem}{\rm (Erd\H{o}s-R\'enyi bound, see \cite{FS} Theorem 2.26)} \label{ER}
Let $\cF=\{F_1,\ldots, F_t\}$ be a family of graphs, and let $c=\max_j\min_{H\subseteq F_j}
\frac{n(H)}{e(H)}$ and $\gamma=\max_j \min_{H\subseteq L_j} \frac{n(H)-2}{e(H)-1}$. Then there exists a positive constant $c_\cF$ depending on $\cF$ such that $ex(n,\cF)>c_\cF n^{2-\gamma}\geq c_\cF n^{2-c}$.
\end{theorem}

Given a graph $H$, let $d(H)$ denote the average degree of $H$. For each positive real $d$,
and positive integer $m$, let $$\cF_{d,m}=\{H: d(H)\geq d, n(H)\leq m\}.$$
Motivated by applications to coding theory and combinatorial number theory \cite{NV},
Verstra\"ete \cite{verstraete} proposed the study of $ex(n,\cF_{d,m})$. Additionally, the study of $ex(n,\cF_{d,m})$ 
may be viewed as a natural generalization of the girth problem.
Indeed, $ex(n,\cF_{2,m})$ is precisely the maximum number of edges in an $n$-vertex graph that does not contain any cycle of length at most $m$. The following lower bound on $ex(n,\cF_{d,m})$ follows immediately from Theorem \ref{ER}.
\begin{proposition}
Let $d\geq 2$ be a real and $m\geq 2$ an integer. Then there exists a positive constant $c$ such that $$ex(n,\cF_{d,m})\geq c n^{2-\frac{m-2}{dm/2-1}}>c n^{2-\frac{2}{d}+\frac{2}{dm}}.$$
\end{proposition}

Verstra\"ete \cite{verstraete} asked the following question.
\begin{question} {\rm \cite{verstraete}} \label{jacques-question}
For each fixed real $d\geq 2$, is it true that there exists a function $\epsilon_d(m)$ such that $\epsilon_d(m)\to 0$ as $m\to \infty$ and $ex(n,\cF_{d,m})=O(n^{2-\frac{2}{d}+\epsilon_d(m)})$?
\end{question}
As the main result of the paper, we give an affirmative answer to Question \ref{jacques-question} whenever $d\geq 2$ is an integer. 
Furthermore, for even integers $d$, we show that the answer is affirmative even when $\cF_{d,m}$ is replaced with the more restrictive family $\cG_{d,m}=\{H: \delta(H)\geq d, n(H)\leq m\}$, where $\delta(H)$ denotes the minimum degree of $H$. Finally, we prove an extension of the cube theorem, which partially answers a question of Pinchasi and Sharir \cite{PS}.
%%%%%%%%%%%%%%%%%%%%%%%%%%%%%%%%%%%%%%%%%%%%%%%%%%%
\section{Overview}

One main idea is to use supersaturation of certain subgraphs $H$ (which we may view as building blocks) to force members of $\cF(d,m)$ when the host graph $G$ is dense enough.
For the even $d=2t$ case, the building blocks we use are $K_{t,t}$'s. For the odd $d=2t+1$  case, the building blocks we use are graphs which we denote by $H_{t,t}$, which is a graph obtained by 
joining two copies of $K_{t,t}$ using a matching. For our supersaturation arguments,
it is more convenient to view it as joining two vertex disjoint 
$t$-matchings in a fashion like $K_{t,t}$
but with edges joined only between vertices in opposite partite sets.

Now, after probing into the idea further, one would realize that supersaturation of $H$ alone will not give us any structure among the copies of $H$
to build members of $\cF_{d,m}$. A second main idea is to introduce some local sparseness, which luckily can be easily accomplished (since otherwise we can already
get certain member of $\cF_{d,m}$). Once we have the local sparseness, we can
apply a random splitting procedure to generate a useful layered structure of the host
graph $G$. Then we use the notion of goodness and the usual Breadth-first search
expansion argument along this pre-designed layered structure to force a member
of $\cF(d,m)$.

Our slight generalization of the cube theorem establishes that $ex(n,H_{t,t})=O(n^{\frac{4t}{2t+1}})$.
The main techniques are centered around analyzing the average behavior of the common neighborhood of a $t$-matching through the count of $C_4$'s and other structures.
We build on Pinchasi and Sharir's new proof \cite{PS} of the cube theorem as well
as the regularization arguments in the original proof of Erd\H{o}s and Simonovits \cite{cube}.

We will pose related questions on both topics in the concluding remarks.

\section{Goodness}

A main notion used in our proofs is that of goodness of a vertex. This notion of goodness
is in part inspired by works in \cite{CFS}, \cite{bukh}, and \cite{Boehnlein-Jiang}.

\begin{definition} \label{goodness}
{\rm
Let $h$ be a positive integer. Let $G$ be a graph with average degree $D$.
We define a vertex $x$ in $G$ to be $(h,1)$-good if $d(x)\geq \frac{1}{3^h}D$.
In general, for $i=2,\ldots, h$, we define a vertex $x$ to be {\it $(h,i)$-good}
if it is $(h,1)$-good and at least half of its neighbors are $(h,i-1)$-good.
A vertex that is not $(h,i)$-good is called {\it $(h,i)$-bad}.}
\end{definition}

\begin{remark}
{\rm 
Let $h$ be a positive integer and $G$ a graph. For each $i=1,\ldots, h$, let
$\cA_i$ denote the set of $(h,i)$-good vertices in $G$ and $\cB_i$ the set of
$(h,i)$-bad vertices in $G$. If follows from induction that any $(h,i)$-good vertex
is also $(h,j)$-good for all $1\leq j\leq i$. So, $\cA_1\supseteq \cA_2\cdots
\supseteq \cA_h$ and $\cB_1\subseteq \cB_2\cdots \subseteq \cB_h$.}
\end{remark}

\begin{lemma} \label{bound-on-goodness}
Let $h$ be a positive integer and $G$ a graph. For each $i\in [h]$, let $\cA_i$ and
$\cB_i$ denote the set of $(h,i)$-good and the set of $(h,i)$-bad vertices, respectively. Then $\sum_{x\in\cB_h} d(x)\leq 
\frac{2}{3} e(G)$ and $\sum_{x\in \cA_h} d(x)\geq \frac{4}{3} e(G)$.
\end{lemma}
\begin{proof}
For each $i=1,\ldots, h$, let $s_i=\sum_{x\in \cB_i} d(x)$. 
We use induction of $i$ to show for each $i=1,\ldots, h$ we have $s_i\leq \frac{2e(G)}{3^{h-i+1}}$.
Let $D$ denote the average degree of $G$. Then $D=\frac{2e(G)}{n}$, where
$n$ is the number of vertices in $G$.
By the definition of $\cB_1$, we have $s_1\leq n \frac{D}{3^h}=\frac{2e(G)}{3^h}$. 
So the claim holds.
Let $2\leq i\leq h$ and suppose that $s_{i-1}\leq 
\frac{2e(G)}{3^{h-i+2}}$. Let $\mu$ denote the number of ordered pairs $(x,x')$
such that $x\in \cB_i\setminus \cB_{i-1}$, $x'\in N(x)$, and that $x'\in \cB_{i-1}$.
Since $x\notin \cB_1$, $x$ is $(h,1)$-good. But $x\in \cB_i$. So by definition,
at least half of the neighbors of $x$ are $(h,i-1)$-bad. Hence
$\mu\geq \sum_{x\in \cB_i\setminus \cB_{i-1}} \frac{1}{2} d(x)$.
So, $\sum_{x\in \cB_i\setminus \cB_{i-1}} d(x)\leq 2\mu$.
On the other hand, if we count the pairs $(x,x')$ by $x'$, then clearly
$\mu\leq \sum_{x'\in \cB_{i-1}} d(x')=s_{i-1}$. Hence 
$\sum_{x\in \cB_i\setminus \cB_{i-1}} d(x)\leq 2\mu \leq 2s_{i-1}$.
Therefore,
$$s_i=\sum_{x\in \cB_i\setminus\cB_{i-1}} d(x) +\sum_{x\in \cB_{i-1}} d(x)
\leq 2s_{i-1}+s_{i-1}=3s_{i-1}\leq \frac{2e(G)}{3^{h-i+1}}.$$
So we have $\sum_{x\in \cB_h} d(x)\leq \frac{2e(G)}{3}$. Since
$\sum_{x\in V(G)} d(x)=2e(G)$ and $\cA_h$ and $\cB_h$ partition $V(G)$,
we have $\sum_{x\in \cA_h} d(x)\geq \frac{4e(G)}{3}$.
This completes the proof.
\end{proof}

We will sometimes use the following lemma to lower bound a binomial coefficient. Other times, we may use more standard approximations.
\begin{lemma} \label{binom-approx}
Let $x,m$ be positive integers where $x\geq m^2$. Then $\binom{x}{m}\geq \frac{x^m}{2m!}$.
\end{lemma}
\begin{proof} 
We have $\binom{x}{m}=\frac{x(x-1)\cdots (x-m+1)}{m!}\geq \frac{x^m}{m!} e^{-\frac{1}{x}}\cdots e^{-\frac{m-1}{x}}\geq \frac{x^m}{m!}e^{-\frac{\binom{m}{2}}{x}}>\frac{x^me^{-\frac{1}{2}}}{m!}>\frac{x^m}{2m!}$.
\end{proof}
%%%%%%%%%%%%%%%%%%%%%%%%%%%%%%%%%%%%%%%%%%%%%%%%%%%%

\section{Even case}

In this section, for convenience, we always assume $t\geq 2$. 
Given a graph $G$ and a set $S$ of vertices in $G$,
we define the {\it common neighborhood} of $S$, denoted by $N^*(S)$, 
to be $N^*(S)=\bigcap_{v\in S} N(v)$. Let $d^*(S)=|N^*(S)|$ and call it 
the {\it common degree} of $S$. 

We need the following  proposition on the supersaturation of $K_{t,t}$'s in dense graphs.
The topic is well-studied, we only give a very rough version of such a supersaturation statement.

\begin{proposition} \label{Ktt-count}
Let $t\geq 2$ be an integer.
Let $G$ be a graph with $n$ vertices and $E\geq t n^{2-\frac{1}{t}}$ edges, where
$n\geq t^2$.
Then the number of $K_{t,t}$'s in $G$ is at least $c_t \frac{E^{t^2}}{n^{2t^2-2t}}$,
where $c_t=\frac{2^{t^2-t-3}}{(t!)^2}$. 
\end{proposition}
\begin{proof}
Let $\lambda$ denote the number of $K_{1,t}$'s in $G$. Then $\lambda=\sum_{x\in V(G)} \binom{d(x)}{t}$. Let $D$ denote the average degree  of $G$. Then $D=\frac{2E}{n}\geq 2tn^{1-\frac{1}{t}}$.
It is easy to see by our assumption that $D\geq t^2$. By convexity and Lemma \ref{binom-approx} we have 
$$\lambda\geq n\binom{D}{t}\geq n \frac{D^t}{2t!}= \frac{n(\frac{2E}{n})^t}{2t!}=
\frac{2^{t-1} E^t}{t!n^{t-1}}.$$ Let $D^*$ denote the average 
common degree of $S$ over all $t$-sets in $G$. Note that $\sum_{S\in \binom{V(G)}{t}}
d^*(S)=\lambda$, as both count the number of pairs $(v,S)$ where $|S|=t$ and $v\in N^*(S)$. So 
$$D^*=\frac{\lambda}{\binom{n}{t}}\geq \frac{\lambda}{n^t/t!}\geq
\frac{2^{t-1} E^t}{t!n^{t-1}}\frac{t!}{n^t} =\frac{2^{t-1} E^t}{n^{2t-1}}.$$ 
Since $E\geq t n^{2-\frac{1}{t}}$, one can check that $D^*\geq t^2$.

Let $\mu$ denote the number of $K_{t,t}$'s in $G$. Then $\mu\geq \frac{1}{2}
\sum_{S\in \binom{V(G)}{t}} \binom {d^*(S)}{t}$. Using convexity and Lemma \ref{binom-approx}, we get 
$$\mu\geq \frac{1}{2} \binom{n}{t} \binom{D^*}{t}
\geq \frac{1}{2}\frac{n^t}{2t!} \frac{(D^*)^t}{2t!}\geq \frac{1}{2}\frac{n^t}{2t!}
\frac{2^{t^2-t}E^{t^2}}{2t! n^{2t^2-t}} =\frac{2^{t^2-t-3}}{(t!)^2} 
\frac{E^{t^2}}{n^{2t^2-2t}}.$$
\end{proof}

For brevity, we call a $t$-uniform hypergraph a {\it  $t$-graph}. A {\it matching} in a hypergraph is a set of
pairwise vertex disjoint edges. 

\begin{lemma} \label{matching}
Let $H$ be a $t$-graph in which each vertex lies in at most $D$ edges. Then $H$ contains a matching of size
at least $e(H)/tD$.
\end{lemma}
\begin{proof}
Let $M$ be a maximum matching in $H$. Then each edge of $H$ must intersect $V(M)$. On the other hand, each vertex in $V(M)$ is contained
in at most $D$ edges of $H$. So $e(H)\leq D|V(M)|=tD|M|$, which yields $|M|\geq e(H)/tD$.
\end{proof}

\begin{lemma} \label{t-set-union}
Let $m,t$ be positive integers. Let $H$ be a $t$-graph with at least $\binom{m}{t}$ edges. There exists a collection
of  edges $E_1,\ldots, E_p$, where $p\leq m-t+1$, such that $|\bigcup_{i=1}^p E_i|\geq m$.
\end{lemma}
\begin{proof}
Let $E_1$ be any edge in $H$.  Let $1\leq i\leq m-t$. Suppose $E_1,\ldots, E_i$ have been chosen. If $|\bigcup_{j=1}^i E_j|\geq m$ then we are done. If $|\bigcup_{j=1}^i E_j|<m$,
then we let  $E_{i+1}$ be any edge of $H$ not completely contained in $\bigcup_{j=1}^i E_i$. Such an edge $E_{i+1}$ exists since $H$ has
at least $\binom{m}{t}$ edges. 
Since each new edge after $E_1$ added to the collection involves at least one new vertex, in at most $m-t+1$ steps, the union of the selected edges will have size at least $m$.
\end{proof}

The following splitting lemma plays a crucial role in our proof of the main theorem
for the even case.
Even though one can get better constants without using this splitting lemma,
the presentation would be much cleaner  using the splitting lemma. 
Recall that by Chernoff's inequality, for a binomially distributed variable variable $X\in Bin(n,p)$ we have $\mP(|X-E(X)|\geq \lambda E(X))\leq 2e^{-\frac{\lambda^2}{3}E(X)}$, as long as $\lambda\leq 3/2$ (see (\cite{JLR} Corollary 2.3).

\begin{definition}
{\rm 
Given a graph $G$, let $K_{t,t}(G)$ denote the auxiliary graph whose vertices are
$t$-sets of $V(G)$ such that two vertices $u,v$ are adjacent if and only if the two $t$-sets they correspond to in $G$ form the two parts of a copy of $K_{t,t}$ in $G$. 
For fixed positive integers $h,i$, where
$h\geq i$, we say that a $t$-set $S$ in $G$ is {\it $(h,i)$-good} in $G$ if the vertex representing it in $K_{t,t}(G)$ is $(h,i)$-good in $K_{t,t}(G)$. Note that, as before,
if $S$ is $(h,i)$-good, then it is also $(h,j)$-good for every $1\leq j\leq i$.
}
\end{definition}

\begin{lemma} \label{splitting}
Let $h,t\geq 2$ be integers and $b,\e$ positive reals, where $b\geq 1$. There is a constant $c=c(h,t,b)$ such that the following holds. Let
$G$ be an $n$-vertex graph with $E\geq c n^{2-\frac{1}{t}+\epsilon}$ edges, where $n$ satistifies  $n^{\e t}>6t\ln (2hn)$. Then there exists a partition of $V(G)$ into sets $L_1,\ldots, L_h$ such that
for every $t$-set $S$ in $G$ and for every $i,j\in [h]$ if $S$ is $(h,i)$-good then $N^*(S)\cap L_j$ contains
at least $b n^{\e t}$ pairwise vertex disjoint $(h,i-1)$-good $t$-sets. Also, some $L_i$ contains an $(h,h)$-good $t$-set.
\end{lemma}
\begin{proof} 
Choose $c$ so that $\frac{c^{t^2}}{3^h}> (4bt^2h^t)^t$.
For convenience,  let $H=K_{t,t}(G)$.
By definition, $n(H)=\binom{n}{t}$. By Lemma \ref{Ktt-count}, with $c_t=\frac{2^{t^2-t-3}}{(t!)^2}$, we have
$$e(H)\geq c_t \frac{E^{t^2}}{n^{2t^2-2t}}\geq c_t \frac{(cn^{2-\frac{1}{t}+\e})^{t^2}}{n^{2t^2-2t}}=\frac{2^{t^2-t-3}c^{t^2}}{(t!)^2}n^{t+\e t^2}.$$
Hence $d(H)=2e(H)/\binom{n}{t}\geq \frac{2^{t^2-t-2}c^{t^2}}{t!} n^{\e t^2}
>\frac{c^{t^2}}{t!} n^{\e t^2}$. 

Let  $S$ be any $(h,i)$-good $t$-set in $G$, where $1\leq i\leq h$. Let $v$ denote the
corresponding vertex in $H$. By definition, $v$ is $(h,i)$-good in $H$.
So, $d_H(v)\geq \frac{d(H)}{3^{h}}\geq \frac{c^{t^2} n^{\e t^2}}{3^h t!}$. 
Let $a=4bt^2h^t$. By our assumption, $\frac{c^{t^2}}{3^h}\geq a^t$.
So, $d_H(v)\geq \binom{an^{\e t}}{t}$.
Note that $N_H(v)$ corresponds to precisely
$\binom{N^*_G(S)}{t}$. Hence $d^*_G(S)=|N^*_G(S)|\geq a n^{\e t}$.
By definition, at least half of the members of $\binom{N^*_G(S)}{t}$ are
$(h,i-1)$-good.
 By Lemma \ref{matching}, among these $t$-sets there exists
a matching $M_S$ of size at least 
$$\frac{1}{2}\binom{d^*(S)}{t}/t\binom{d^*(S)-1}{t-1}
\geq \frac{d^*(S)}{2t^2}\geq \frac{a n^{\e t}}{2t^2}=2bh^t n^{\e t}.$$  
We have shown that for any $(h,i)$-good $t$-set $S$ in $G$, 
we can fix a matching $M_S$
of $(h,i-1)$-good $t$-sets in $N^*_G(S)$ of size at least $2bh^t n^{\e t}$.

Now, independently and uniformly at random assign a color from $\{1,\ldots, h\}$ to each vertex
of $G$. Fix any $i\in [h]$ and any $(h,i)$-good (but not $(h,i+1)$-good if $i\leq h-1$)
$t$-set $S$. Let $X_{S,j}$ count the number of 
$(h,i-1)$-good sets $T$ in $M_S$  whose vertices all received color $j$.
Since edges in $M_S$ are pairwise vertex disjoint and each edge is monochromatic in $j$
with probability $\frac{1}{h^t}$, $X_{S,j}\in Bin(|M_S|,\frac{1}{h^t})$. So,
$\mE(X_{S,j})=\frac{|M_S|}{h^t}$ and by the Chernoff bound, $\mP(X_{S,j}<\frac{|M_S|}{2h^t})\leq 2e^{-\frac{1}{12} \frac{|M_S|}{h^t}}
\leq 2e^{-\frac{b n^{\e t}}{6}}$. 
 Hence,
$\mP(\exists S, j: X_{S,j}<\frac{|M_S|}{2h^t})<2hn^t e^{-\frac{b n^{\e t}}{6}} 
<2hn^t e^{-\frac{n^{\e t}}{6}} <\frac{1}{h^{t-1}}$. 
where one can check that the last inequality holds when
 $n^{\e t}> 6t\ln (2hn)$, 

Next, note that by Lemma \ref{bound-on-goodness}, $H$ contains at least one
$(h,h)$-good vertex. Hence, $G$ has at least one $(h,h)$-good $t$-set $U$. The probability that  $U$ is monochromatic
is $\frac{h}{h^t}=\frac{1}{h^{t-1}}$. This combined with earlier discussion shows that 
there exists a coloring for which $\forall S,j$ we have
$X_{S,j}\geq \frac{|M_S|}{2h^t}\geq bn^{\e t}$ and  $U$ is monochromatic. 
For each $i\in [h]$, let $L_i$ denote color class $i$. The claim follows.
\end{proof}

Now we are ready to prove our main theorem for the even case. Given a graph $H$,
let $rad(H)$ denote its radius.

\begin{theorem} \label{even} {\bf (Main theorem for even case)}
Let $r,t\geq 2$ be integers.
There is a constant $\alpha=\alpha(r,t)$ such that the following holds.
Let $G$ be an $n$-vertex graph with $e(G)\geq \alpha n^{2-\frac{1}{t}+\frac{1}{r}}$ edges, where $n^{\frac{t}{r}}>6t\ln (2nr)$ and $n\geq t^2$.  Then $G$ contains a subgraph $G^*$ with $\delta(G^*)\geq 2t, rad(G^*)\leq r$ and $n(G^*)<rt^2+rt$.
\end{theorem}
\begin{proof}
Apply Lemma \ref{splitting}, with $\e=\frac{1}{r}$, $h=r$, and $b=\binom{2t}{t}$,
and let $\alpha$ be the constant $c$ returned by the lemma. By the lemma,
there exists a partition of $V(G)$ into $L_1,\ldots, L_r$ such that 
for every $t$-set $S$ in $G$ and for every $i,j\in [r]$ if $S$ is $(r,i)$-good then $N^*(S)\cap L_j$ contains
a collection $\cC_j(S)$ of at least $\binom{2t}{t} n^ {t/r}$ pairwise vertex disjoint $(r,i-1)$-good $t$-sets. Furthermore, some $L_i$ contains
an $(r,r)$-good $t$-set. By relabeling if necessary, we may assume that $L_1$ contains an $(r,r)$-good $t$-set $U_0$. For each vertex $x$ in $K_{t,t}(G)$, let
$S(x)$ denote the $t$-set in $G$ that $x$ represents. 

Now we define an auxiliary digraph $H$ with $V(H)\subseteq V(K_{t,t}(G))$
together with a partition $B_0,B_1,\ldots, B_r$ of $V(H)$ as follows. 
Let $B_0$ consist of the single vertex $u$ in $K_{t,t}(G)$ representing $U_0$.
Let $B_1$ be the set of vertices in $K_{t,t}(G)$ representing $t$-sets in  
$\cC_1(U_0)$. For each $i\in \{2,\ldots, r\}$, let $B_i$ the set of vertices in $K_{t,t}(G)$ representing $(r,r-i)$-good $t$-sets in $L_i$.
(Here, we define every $t$-set in $G$ to be $(r,0)$-good.)
Next, for each $i\in [r-1]$ and each vertex $x\in B_i$,  we add arcs from $x$ to all the vertices in $B_{i+1}$ that represent $t$-sets in $\cC_{i+1}(S(x))$. This defines the digraph $H$.  

By our assumptions about the $L_i$'s,  for each $i\in [r-1]\cup\{0\}$, each vertex in $B_i$ has at least $\binom{2t}{t} n^{t/r}$
out-neighbors in $B_{i+1}$. Now, grow a breadth-first search out-tree $T$ in $H$ from $u$. For each $i$, let $D_i=V(T)\cap B_i$.
For each $i\in [r-1]\cup \{0\}$, by our assumption, $D_i$ sends out at least $|D_i|\binom{2t}{t}n^{\e t}$ edges into $D_{i+1}$. We consider two cases.

\medskip

\noindent {\bf Case 1.} For some $i\in [r]$, $D_i$ contains a vertex $y$ that lies in the out-neighborhoods of at least $\binom{2t}{t}$ different vertices $x_1,\ldots, x_{\binom{2t}{t}}$
in $D_{i-1}$.  

\medskip

Since $S(x_1), \ldots, S(x_{\binom{2t}{t}})$ are distinct $t$-sets, by Lemma \ref{t-set-union}, there exist a collection of $t+1$ of them whose union
have size at least $2t$. Without loss of generality, suppose that $|\bigcup_{\ell=1}^{t+1} S(x_\ell)|\geq 2t$.
Let $v$ denote the closest common ancestor of $x_1,\ldots, x_{t+1}$ in $T$. Suppose $v\in D_j$.
Let $T'$ be the subtree of $T$ consisting of the directed paths from $v$ to $\{x_1,\ldots, x_{t+1}\}$. Let $F$ be the union of $T'$ and the edges $x_1y, \cdots, x_{t+1}y$.
Let $G^*$ be the subgraph of $G$ induced by $\bigcup_{x\in V(F)} S(x)$. We show that
$G^*$ has minimum degree at least $2t$. For each $k=j,j+1,\ldots, i$, let $A_k=\bigcup_{x\in V(F)\cap D_k} S(x)$.  Then $A_k\subseteq L_k$, unless $k=0$
in which case $A_0=U_0$. Using this, one can check that 
$A_j, A_{j+1},\ldots, A_i$ are pairwise vertex disjoint in $G$. 
We need to show that for each $k=j,j+1,\ldots, i$ and any $x\in A_k$ we have $d_{G^*}(x)\geq 2t$.  Note that $A_j=S(v)$ and $A_{i+1}=S(y)$. Since $v$ is the closest common ancestor of $x_1,\ldots, x_{t+1}$
in $T$, $v$  has at least two children in $T'$.  Let $a,b$ denote two of the children of $v$ in $T'$. By the definition of $H$, the out-neighborhood of $v$ in $H$ corresponds to $\cC_{j+1}(S(v))$, which consists
of pairwise vertex disjoint $t$-sets in $L_{j+1}$. Hence $S(a)\cap S(b)=\emptyset$.
Since $va, vb\in E(H)$, $N^*(S(v))$ contains $S(a)$ and $S(b)$. Hence
each vertex in $A_j=S(v)$ has degree at least $2t$ in $G^*$. Next, let $k\in \{j+1,\ldots, i-1\}$. Let $x\in V(F)\cap D_k$. Then
$x$ has an in-neighbor $x^-$ in $D_{k-1}$ and at least one out-neighbor $x^+$ in $D_{k+1}$. Since $x^-x, xx^+\in E(H)$,
by definition, $G$ contains a copy of $K_{t,t}$ between $S(x^-)$ and $S(x)$ and a copy of $K_{t,t}$ between $S(x)$ and $S(x^+)$,
both of which are in $G^*$. Since $S(x^-), S(x), S^+(x)$ are pairwise disjoint due to the disjointness of $A_j,A_{j+1},\ldots, A_i$, each vertex in $S(x)$ has degree 
at least $2t$ in $G^*$. This shows that for each $x\in A_k$, $d_{G^*}(x)\geq 2t$.  Finally, consider $A_i=S(y)$. Since $x_1y,\ldots, x_{t+1}y\in E(H)$,
each vertex in $S(y)$ is adjacent in $G^*$ to all of $\bigcup_{\ell=1}^{t+1} S(x_\ell)$. By our earlier discussion, $|\bigcup_{\ell=1}^{t+1} S(x_\ell)|\geq 2t$.
Hence each vertex in $S(y)=A_i$ has degree at least $2t$ in $G^*$. Now we have found a subgraph $G^*$ of $G$ with minimum degree at least $2t$. The number of vertex in $T'$ is at most
$(r-1)(t+1)+1$ since it has $t+1$ leaves and has height at most $r-1$. So $n(F)\leq (r-1)(t+1)+2<rt+r$ and thus $n(G^*)\leq rt^2+rt$. Also, $rad(G^*)\leq r$.

\medskip

\noindent {\bf Case 2.} For each $i\in [r]$ every vertex in $D_i$ lies in the out-neighborhoods of fewer than $\binom{2t}{t}$ vertices of $D_{i-1}$.

\medskip

For each $i\in [r]$, since  $D_{i-1}$ sends out at least $|D_{i-1}|\binom{2t}{t} n^{t/r}$ edges into $D_i$ and each vertex in $D_i$ receives 
fewer than $\binom{2t}{t}$ of these edges, we have $|D_i|\geq  n^{t/r}|D_{i-1}|$. 
This yields $|D_r|\geq [n^{t/r}]^r=n^t>\binom{n}{t}$, which is impossible since vertices in $D_r$ correspond
to distinct $t$-sets in $G$ and there are only $\binom{n}{t}$ distinct $t$-sets in $G$. 
\end{proof}

Applying Theorem \ref{even} with $r=\fl{\frac{m}{2t^2}}$, we answer Question \ref{jacques-question} in the stronger form for even $d$.

\begin{proposition} 
Let $t,m\geq 2$ be integers. We have $ex(n,\cG_{2t,m})=O(n^{2-\frac{1}{t}+\frac{2t^2}{m}})=O(n^{2-\frac{2}{2t}+\frac{2t^2}{m}})$.
\end{proposition}

%%%%%%%%%%%%%%%%%%%%%%%%%%%%%%%%%%%%%%%%%%%%%%%%%%

\section{Odd case}

In this section, unless otherwise specified, we allow $t=1$.
\begin{definition}
{\rm Let $s,t$ positive integers.  Let $M$ be an $s$-matching 
$x_1y_1,x_2y_2,\ldots, x_sy_s$, and $N$ a $t$-matching $x'_1y'_1,\ldots, x'_ty'_t$
where $M$ and $N$ are vertex disjoint.
Let $H_{s,t}$ be obtained from $M$ and $N$ 
by adding edges $x_ix'_j$ and $y_iy'_j$ over all
$i\in [s]$ and $j\in [t]$.   We call $M$ and $N$ the {\it two parts} of $H_{s,t}$.
Equivalently, $H_{s,t}$ can be obtained as follows: start with 
a copy $B_x$ of $K_{s,t}$ with parts $\{x_1,\ldots, x_s\}$
and $\{x'_1,\ldots, x'_t\}$ and another copy $B_y$ of $K_{s,t}$ with parts $\{y_1,\ldots, y_s\}$ and $\{y'_1,\ldots,
y'_t\}$ and then add a $(s+t)$-matching $x_iy_i, x'_j,y'_j$, for all $i\in [s]$ and $j\in [t]$.
}
\end{definition}

Note that $H_{1,1}$ is the four-cycle $C_4$ and $H_{2,2}$ is the 3-dimensional cube $Q_3$. A well-known result of Erd\H{o}s and Simonovits \cite{cube} shows that $ex(n,Q_3)=O(n^{\frac{8}{5}})$.
Pinchasi and Sharir \cite{PS} gave a new proof of this result and also obtained the following.

\begin{theorem} {\bf \cite{PS}} \label{PS-theorem}
Let $2\leq s\leq t$ be positive integers and let $G$ be a graph on $n$ vertices which does not contain a copy of
$H_{s,t}$ and also does not contain a copy of $K_{s+1,s+1}$. Then $G$ has at most $O(n^{\frac{4s}{2s+1}})$ edges.
\end{theorem}

Equivalently, Theorem \ref{PS-theorem} establishes that $ex(n,\{H_{s,t}, K_{s+1,s+1}\})=O(n^{\frac{4s}{2s+1}})$.
Pinchasi and Sharir \cite{PS} asked if Theorem \ref{PS-theorem} can be strengthened to $ex(n,H_{s,t})=O(n^{\frac{4s}{2s+1}})$.
In Section \ref{Htt-section} we give an affirmative answer to the question for the case $s=t$. That is,
we show that $ex(n,H_{t,t})=O(n^{\frac{4t}{2t+1}})$. This provides a  generalization of the cube theorem of Erd\H{o}s and Simonovits.

In this section, we first establish supersaturation of $H_{t,t}$'s in the absence of $K_{t+1,q}$'s. Then we use supersaturation, splitting, and expansion arguments to establish our main theorem for the odd case. Arguments in this section are much more technical than in the previous one, as we will be analyzing interactions between pairs of $t$-matchings, rather than between two $t$-sets of vertices. We start our supersaturation arguments by counting $t$-matchings.
Counting matchings of a fixed size in a graph is a well-studied topic. For our purposes,  however, we will only need the following very crude bound. We consider a $t$-matching
to be an unordered set of $t$ disjoint edges.

\begin{lemma} \label{matching-count}
Let $G$ be a graph with maximum degree $d$ and $E$ edges, where $E\geq 4dt$.
Then the number of $t$-matchings in $G$ is at least $\frac{E^t}{2^tt!}$. Also, if 
$E\geq 4dt^2$, then the number of $t$-matchings in $G$ is at least $\frac{1}{2} \frac{E^t}{t!}$.
\end{lemma}
\begin{proof}
Consider selecting $t$ disjoint edges $e_1,\ldots, e_t$ greedily as follows. First we select an arbitrary edge to be $e_1$.
Then delete the all the edges of $G$ that are incident to $e_1$; there are at most $2(d-1)$ of them. Then we select
an arbitary remaining edge to be $e_2$, and deleting edges incident to $e_2$, and etc. The number of different lists
$e_1,\ldots, e_t$ we produce this way is at least $\mu=E(E-2d)(E-4d)\ldots [E-2d(t-1)]\geq (\frac{E}{2})^t$. So the number of different
sets $\{e_1,\ldots, e_t\}$ is at least $\frac{E^t}{2^t t!}$.
Next, suppose $E\geq 4dt^2$. Then $\mu=E^t\Pi_{i=1}^{t-1} (1-\frac{2di}{E})
\geq E^t \Pi_{i=1}^{t-1} e^{-2\frac{2di}{E}}\geq E^t e^{-\frac{2dt^2}{E}}
\geq \frac{1}{2}E^t$. Hence the number of different $t$-sets $\{e_1,\ldots, e_t\}$ 
is at least $\frac{1}{2} \frac{E^t}{t!}$.
\end{proof}

Next, we establish supersaturation properties of  $H_{1,t}$'s in bipartite graphs.
The symmetric version is implied by Theorem 4 of \cite{cube} and the asymmetric
version is implicit in \cite{cube}. However, for the purpose of the next section, we
need an explicit asymmetric version. Since the arguments are standard convexity argugments and are short, we include them
for completeness.  We follow arguments used in \cite{PS} in the next two lemmas.

\begin{lemma} \label{2-claw-c4}
Let $G$ be an $n$-vertex biparite graph with a bipartiton $(A,B)$. 
Suppose $G$ has $E\geq n^{3/2}$ edges. Let $W_A$ and $W_B$ denote the
number of $K_{1,2}$'s in $G$ centered in $A$ and in $B$, respectively.
Let $S$ denote the number of $C_4$'s in $G$. Then
$W_A\geq \frac{E^2}{4|A|}$, $W_B\geq \frac{E^2}{4|B|}$, 
$S\geq \frac{W_A^2}{2|B|^2}$, and $S\geq \frac{W_B^2}{2|A|^2}$.
In particular, we have $S\geq \frac{E^4}{32|A|^2|B|^2}$.
\end{lemma}
\begin{proof}
For any real $x\geq 2$ we have $\binom{x}{2}=\frac{x(x-1)}{2}\geq \frac{x^2}{4}$.
Let $d_A=\frac{E}{|A|}$ denote the average degree in $G$ of vertices in $A$. Then $d\geq 2\sqrt{n}$. By convexity, we have
\begin{equation}\label{WA}
W_A=\sum_{x\in A} \binom{d(x)}{2}\geq |A|\binom{d_A}{2}\geq \frac{|A| d_A^2}{4}\geq \frac{E^2}{4|A|}.
\end{equation}
By a similar argument, we have $W_B\geq \frac{E^2}{4|B|}$.
For each pair $u,v$ of vertices, let $d(u,v)$ denote the number of common neighbors of $u$ and $v$.
Let $d^*_B$ denote the average of $d(u,v)$ over all pairs $u,v$ in $B$.
Note that $\sum_{u,v\in B} d(u,v)=W_A$.  Hence $d^*_B=\frac{W_A}{\binom{|B|}{2}}\geq \frac{E^2}{2|A||B|^2}\geq 2$, where the last inequality
follows from $E\geq n^{3/2}$ and $n=|A|+|B|$.
Now, using convexity, we have 
\begin{equation} \label{S}
S\geq \sum_{u,v\in B}\binom{d(u,v)}{2}\geq 
\binom{|B|}{2} \binom{d^*_B}{2}
\geq \frac{W_A^2}{4\binom{|B|}{2}}\geq \frac{W_A^2}{2|B|^2}.
\end{equation}
Similarly, we have $S\geq \frac{W_B^2}{2|A|^2}$. By \eqref{S} and \eqref{WA}, we have $S\geq \frac{W_A^2}{2|B|^2}\geq \frac{E^4}{32|A|^2|B|^2}$.
\end{proof}

\begin{lemma} \label{H1t-count}
Let $t$ be a positive integer.
Let $G$ be an $n$-vertex bipartite graph with a bipartition $(A,B)$.
Suppose $G$ has $E\geq 4\sqrt{2t} n^{3/2}$ edges. Then the number
of $H_{1,t}$'s in $G$ is at least $\frac{1}{2^{5t+2} t!} \frac{E^{3t+1}}{|A|^{2t}|B|^{2t}}$.
\end{lemma}
\begin{proof}
For each edge $e=xy$, where $x\in A$ and $y\in B$, let
$X_e=N(y)\setminus\{x\}$ and $Y_e=N(x)\setminus \{y\}$. Let $G_e$ denote the subgraph of $G$ induced by
$X_e\cup Y_e$. Let $V_e$ and $E_e$ denote the  number of vertices and edges in $G_e$, respectively.
We call an edge $e$ {\it good}  if $E_e\geq 8t V_e$ and {\it bad} if $E_e<8t V_e$.
Let $\cE_1$ denote the set of good edges in $G$ and $\cE_2$ the set of bad edges in $G$.

{\bf Claim 1.} We have $\sum_{e\in E(G)} V_e \leq \frac{1}{16 t} \sum_{e\in E(G)} E_e$.

\medskip
 
{\it Proof of Claim 1.} 
Let $W$ denote the number of $K_{1,2}$'s in $G$ and $S$ the number of $C_4$'s in $G$.
Then $\sum_{e\in E(G)} V_e= 2W$ and $\sum_{e\in E(G)} E_e=4S$.
Suppose for contradiction that $\sum_{e\in E(G)} V_e > \frac{1}{16 t} \sum_{e\in E(G)} E_e$. Then $\sum_{e\in E(G)}  E_e< 16t \sum_{e\in E(G)} V_e$, or equivalently, $4S\leq 32tW$.
Hence $S\leq 8tW$.  Let $W_A, W_B$ denote the number of $K_{1,2}$'s centered in $A$ and $B$, respectively in $G$.
Without loss of generality, suppose $W_A\geq W_B$. We have $S\leq 8tW\leq 16tW_A$.
On the other hand, by Lemma \ref{2-claw-c4}, we have $S\geq \frac{W_A^2}{2|B|^2}$.
Thus, we have $\frac{W_A^2}{2|B|^2}\leq 16tW_A$. Solving for $W_A$ 
yields $W_A\leq 32t|B|^2$. On the other hand, by Lemma \ref{2-claw-c4},
we also have $W_A\geq \frac{E^2}{4|A|}$. Hence $\frac{E^2}{4|A|}\leq 32t|B|^2$, which yields $E\leq 8\sqrt{2t}|A|^\frac{1}{2} |B|\leq 4\sqrt{2t}n^{3/2}$, contradicting our assumption about $G$.
\qed

\medskip

Now, by Claim 1 and the definition of $\cE_2$, we have
$$\sum_{e\in \cE_2} E_e\leq 8t \sum_{e\in E(G)} V_e\leq \frac{1}{2}\sum_{e\in E(G)} E_e.$$
Hence,
\begin{equation} \label{lower-on-good}
\sum_{e\in \cE_1} E_e\geq \frac{1}{2} \sum_{e\in E(G)} E_e=2S.
\end{equation}

By Lemma \ref{2-claw-c4}, $S\geq \frac{E^4}{32|A|^2|B|^2}$.
Hence,
\begin{equation} \label{lower-on-good2}
\sum_{e\in \cE_1} E_e\geq \frac{E^4}{16|A|^2|B|^2}.
\end{equation}
For each $e\in \cE_1$, since $E_e\geq 8tV_e\geq 4t\Delta(G_e)$, by Lemma \ref{matching-count},
$G_e$ contains at least $\frac{(E_e)^t}{2^t t!}$ different $t$-matchings.  Let $\lambda$ denote the number
of $H_{1,t}$'s in $G$. Then  $\lambda\geq \frac{1}{4}\sum_{e\in \cE_1} \frac{(E_e)^t}{2^t t!}=\frac{1}{2^{t+2} t!} \sum_{e\in \cE_1} (E_e)^t$.
Using convexity and \eqref{lower-on-good2}, we have
$$\lambda\geq \frac{1}{2^{t+2} t!} \frac{(\sum_{e\in \cE_1} E_e)^t}{|\cE_1|^{t-1}} \geq
\frac{1}{2^{t+2} t!}\left(\frac{E^4}{16A|^2|B|^2}\right)^t/E^{t-1}\geq \frac{1}{2^{5t+2} t!} \frac{E^{3t+1}}{|A|^{2t}|B|^{2t}}.$$
\end{proof}

Next, we establish supersaturation of $H_{t,t}$'s in $K_{t+1,q}$-free graphs.
The reason for the extra assumption of $K_{t+1,q}$-freeness is (1) it simplifies the arguments and (2) it is needed for a later splitting process. (For the splitting process
to work, one needs some "local spareness".)

\begin{lemma} \label{Htt-saturation}
Let $t,q$ be positive integers.
Let $G$ be an $n$-vertex  $K_{t+1,q}$-free bipartite graph with $E\geq 12qtn^{\frac{4t}{2t+1}}$ edges.
Then $G$ contains at least $c'_t\frac{E^{2t^2+2t}}{n^{4t^2}}$copies of $H_{t,t}$, where 
$c'_t=\frac{1}{2^{5t^2+4t+1}(t!)^{t+1}}$
\end{lemma}
\begin{proof} 
Let $(A,B)$ be a bipartition of $G$.  Let $M$ be a $t$-matching in $G$. 
Let $X_M=N^*(B\cap V(M))\setminus V(M)$ and $Y_M=N^*(A\cap V(M))\setminus V(M)$.
Let $G_M$ denote the subgraph of $G$ induced by $X_M\cup Y_M$. Then $G_M$ is bipartite with a bipartition $(X_M,Y_M)$.
Let $E_M$ denote the number of edges in $G_M$.
Suppose first that $G_M$ contains a vertex $x$ of degree 
at least $q$. Without loss of generality, suppose $x\in X_M$. Let $y_1,\ldots, y_q\in Y_M$ denote $q$ of the neighbors of
$x$ in $G_M$. Then by the definition of $Y_M$, each $y_i$ is adjacent to all of $V(M)\cap A$. Now, we obtain a copy of
$K_{t+1,q}$ with parts $(V(M)\cap A)\cup \{x\}$ and $\{y_1,\ldots, y_q\}$, contradicting that $G$ is $K_{t+1,q}$-free.
Hence $G_M$ has maximum degree less than $q$. Let's call $M$ {\it good} if $E_M\geq 4qt$ and call $M$ {\it bad} otherwise.
For good $M$'s, by Lemma \ref{matching-count}, $G_M$ contains at least 
$\frac{(E_M)^t}{2^tt!}$ many $t$-matchings. In other words, each good $t$-matching $M$ forms a $H_{t,t}$ with
at least $\frac{(E_M)^t}{2^tt!}$ many $t$-matchings.

Let $\cM$ denote the set of all $t$-matchings in $G$.
Let $\cM_1$ denote the set of good $t$-matchings and $\cM_2$ the set of bad $t$-matchings in $G$.
Let $\mu$ denote the number of $H_{t,t}$'s in $G$. By our discussion, 
\begin{equation} \label{mu-bound}
\mu\geq \frac{1}{2} \sum_{M\in \cM_1}
\frac{(E_M)^t}{2^t t!}=\frac{1}{2^{t+1}t!} \sum_{M\in \cM_1} (E_M)^t.
\end{equation}

Let $\lambda=\sum_{M\in \cM} E_M$. Note that $\lambda$ counts the number of
$H_{1,t}$'s in $G$. By Lemma \ref{H1t-count}, we have 
$$\lambda=\sum_{M\in \cM} E_M \geq \frac{1}{2^{5t+2} t!} \frac{E^{3t+1}}{n^{4t}}.$$
Let $\lambda_1=\sum_{M\in \cE_1} E_M$ and $\lambda_2=\sum_{M\in \cE_2} E_M$. Then $\lambda=\lambda_1+\lambda_2$.
By the definition of $\cE_2$, $\lambda_2\leq 4qt |\cE_2|\leq 4qt E^t$.
On the other hand, using $E\geq 12qt n^{\frac{4t}{2t+1}}$ and 
$\lambda\geq E^t\cdot \frac{E^{2t+1}}{2^{5t+2} t! n^{4t}}$, we can show
that $\lambda\geq 8qtE^t$.  Hence, $\lambda_1\geq \frac{1}{2}\lambda$. So,
\begin{equation} \label{E1bound}
\sum_{M\in \cM_1} E_M\geq \frac{1}{2}\frac{1}{2^{5t+2} t!} \frac{E^{3t+1}}{n^{4t}}.
\end{equation}

Now, by \eqref{mu-bound}, \eqref{E1bound}, and convexity, we have
$$\mu\geq \frac{1}{2^{t+1} t!} \frac{(\sum_{M\in \cM_1} E_M)^t}{(\cM_1)^{t-1}}\geq \frac{1}{2^{t+1} t!}  \frac{(\sum_{M\in \cM_1} E_M)^t}{(E^t)^{t-1}}
\geq \frac{1}{2^{5t^2+4t+1}(t!)^{t+1}} \frac{E^{2t^2+2t}}{n^{4t^2}}.$$
\end{proof}

\begin{lemma} \label{disjoint-matching}
Let $G$ be a graph with $E$ edges and maximum degree at most $q$. Let $\cM$ be the collection of all the $t$-matchings in $G$ and $\cM'\subseteq \cM$ with $|\cM'|\geq \frac{1}{2}|\cM|$. Then $\cM'$ contains at least $\frac{E}{qt^3 2^{t+2}}$ vertex disjoint $t$-matchings. 
\end{lemma}
\begin{proof}
By Lemma \ref{matching-count}, $|\cM|\geq \frac{E^t}{2^tt!}$. 
Let $\cM''$ be a maximum collection of edge-disjoint members of $\cM'$ (recall
that each member of $\cM'$ is a $t$-matching in $G$). Let $L$ denote the set of edges
of $G$ that are contained in the members of $\cM''$. Then $|L|=t|\cM''|$. 
Since $\cM''$ is maximum, each member of $\cM'$ must contain an edge in $L$.
On the other hand, each edge in $L$ clearly lies in fewer than $\frac{E^{t-1}}{(t-1)!}$
members of $\cM'$. Hence, $|\cM'|\leq |L|\frac{E^{t-1}}{(t-1)!}=t|\cM''|\frac{E^{t-1}}{(t-1)!}$. Therefore, 
$$|\cM''|\geq \frac{|\cM'|}{ tE^{t-1}/(t-1)!}\geq  \frac{(1/2)|\cM|}{ tE^{t-1}/(t-1)!}
\geq \frac{(1/2) E^t/2^t t!}{tE^{t-1}/(t-1)!}=\frac{E}{t^2 2^{t+1}}.$$
Now since $G$ has maximum degree at most $q$ and members of $\cM''$ are edge-disjoint,  each vertex in $G$ lies in at most $q$ members of $\cM''$. So each member
of $\cM''$ shares a vertex with fewer than $2tq$ other members of $\cM''$. By a greedy algorithm, one can build a subcollection $\cM'''$ of vertex disjoint members of $\cM''$
with $|\cM'''|\geq |\cM''|/2tq\geq \frac{E}{qt^3 2^{t+2}}$.
\end{proof}

Now we develop a splitting lemma for the odd case. Given a positive integer $t$
and a graph $G$, we let $H_{t,t}(G)$ denote the auxiliary graph whose vertices
are $t$-matchings in $G$ such that two vertices $u,v$ are adjacent in $H_{t,t}(G)$
if and only if the two $t$-matchings they correspond to in $G$ form the two
parts of a copy of $H_{t,t}$ in $G$. Given positive integers $h\geq i\geq 1$,
we say that a $t$-matching $M$ is {\it $(h,i)$-good} in $G$ if the vertex in $H_{t,t}(G)$
that corresponds to $M$ is $(h,i)$-good in $H_{t,t}(G)$. If $G$ is bipartite with a bipartition $(A,B)$ and $M$ is a matching in $G$, then as before, let $X_M=N^*(V(M)\cap B)\setminus V(M)$ and $Y_M=N^*(V(M)\cap A)\setminus V(M)$
and let $G_M$ denote the subgraph of $G$ induced by $X_M\cup Y_M$.

\begin{lemma} \label{splitting2}
Let $h,q,t$ be positive integers and $b,\e$ positive reals, where $b\geq 1$. There is a constant $c=c(h,q,t,b)$ such that following holds.
Let $G$ be an $n$-vertex $K_{t+1,q}$-free bipartite graph with 
$E\geq c n^{\frac{4t}{2t+1}+\epsilon}$ edges, where
$n^{\e (2t+1)}> 12t\ln (h^2n)$. Then there exists a partition of $V(G)$ into sets $L_1,\ldots, L_h$ such that for every $t$-matching $M$ in $G$ and for every $i,j\in [h]$ if $M$ is $(h,i)$-good then $L_j$ contains
at least $b n^{\e t}$ pairwise vertex disjoint $(h,i-1)$-good $t$-matchings
in $G_M$. Furthermore, some $L_i$ contains an $(h,h)$-good $t$-matching.
\end{lemma}
\begin{proof}
We will specify the choice of $c$ later in the proof.
For convenience, let $H=H_{t,t}(G)$. By Lemma \ref{Htt-saturation}, 
$e(H)\geq c'_t  \frac{E^{2t^2+2t}}{n^{4t^2}}$. Clearly, $n(H)\leq E^t$.
Hence $d(H)\geq c'_t  \frac{E^{2t^2+t}}{n^{4t^2}}\geq c'_t c^{2t^2+t} n^{\e(2t^2+t)}$,
where the last inequality follows from $E\geq  c n^{\frac{4t}{2t+1}+\epsilon}$.
Let  $M$ be any $(h,i)$-good $t$-matching  in $G$, 
where $1\leq i\leq h$, let $v$ denote the corresponding vertex in $H$.
Since $v$ is $(h,i)$-good in $H$, by definition, $d_H(v)\geq \frac{d(H)}{3^{h}}\geq \frac{c'_t c^{2t^2+t} n^{\e (2t^2+t)}}{3^h}$. Note that $N_H(v)$ corresponds to the collection $\cM$ of all the $t$-matchings in $G_M$. So, $|\cM|=d_H(v)$. Let $\cM'$ denote the set of $(h,i-1)$-good matchings in $G_M$.  
Since $M$ is $(h,i)$-good, by definition,
$|\cM'|\geq \frac{1}{2} |\cM|$. Note also that since $G$ is $K_{t+1,q}$-free,
$G_M$ has maximum degree less than $q$. 
Let $E_M$ denote the number of edges in $G_M$. Trivially, $|\cM|\leq (E_M)^t/t! $.
So 
$$E_M\geq (t!|\cM|)^{1/t}=(t! d_H(v))^{1/t}>[d_H(v)]^{1/t}\geq (c'_t/3^h)^{1/t}c^{2t+1} n^{\e(2t+1)}.$$
By choosing $c$ to be large enough, we can ensure that $E_M\geq qt^3 2^{t+3} b h^{2t} n^{\e(2t+1)}$.
By Lemma \ref{disjoint-matching}, 
$\cM'$ contains at least $\frac{E_M}{qt^3 2^{t+2}}\geq 2bh^{2t} n^{\e(2t+1)}$ vertex disjoint members.
We have thus shown that for each $(h,i)$-good $t$-matching $M$ in $G$, 
we can fix a collection $\cC_M$ of at least $2bh^{2t} n^{\e (2t+1)}$ vertex disjoint 
$(h,i-1)$-good $t$-matchings in $G_M$.

Now, independently and uniformly at random assign a color from $\{1,\ldots, h\}$ to each vertex of $G$. Fix any $i\in [h]$ and any $(h,i)$-good (but not $(h,i+1)$-good if $i\leq h-1$) $t$-matching $M$. Let $X_{M,j}$ count the number of 
$(h,i-1)$-good $t$-matchings $T$ in $M_S$  in which all the vertices of $T$ are colored $j$. Since the $t$-matchings in $\cC_M$ are pairwise vertex disjoint, 
$X_{M,j}\in Bin(|\cC_M|,\frac{1}{h^{2t}})$. So,
$E(X_{M,j})=\frac{|\cC_M|}{h^{2t}}$ and by the Chernoff bound, $\mP(X_{M,j}<\frac{|\cC_M|}{2h^{2t}})\leq 2e^{-\frac{1}{12} \frac{|\cC_M|}{h^{2t}}} \leq e^{-\frac{b}{6} n^{\e (2t+1)}}$, using $|\cC_M|\geq 2bh^{2t}n^{\e(2t+1)}$. 
Hence,
$\mP(\exists M, j: X_{M,j}<\frac{|\cC_M|}{2h^t})<2hn^t e^{-\frac{b}{6} n^{\e (2t+1)}}<2hn^t e^{-\frac{1}{6} n^{\e (2t+1)}}< \frac{1}{h^{2t-1}}$,
where one can check that the last inequality holds when 
$n^{\e (2t+1)}> 12t\ln (h^2n)$.
Next, note that by Lemma \ref{bound-on-goodness}, $H$ contains at least one
$(h,h)$-good vertex. Hence, $G$ has at least one $(h,h)$-good $t$-matching $M_0$. The probability that all the vertices in $M_0$ have received the same color
is $\frac{h}{h^{2t}}=\frac{1}{h^{2t-1}}$. This combined with earlier discussion shows that 
there exists a coloring for which $\forall M,j$ we have
$X_{M,j}\geq \frac{|\cC_M|}{2h^{2t}}\geq b n^{\e(2t+1)}$ and  that $M_0$
is monochromatic. For each $i\in [h]$, let $L_i$ denote color class $i$. The claim follows.
\end{proof}

\begin{theorem} \label{odd} {\bf (Main theorem for odd case)}
Let $r,t$ be positive integers. There is a constant $\beta=\beta(r,t)$ such that the following holds. Let $G$ be an $n$-vertex graph with $e(G)\geq \beta n^{\frac{4t}{2t+1}+\frac{1}{r}}$ edges, where $n^{\e (2t+1)}> 12t\ln (h^2n)$
and $n>t!$.  Then $G$ contains a subgraph $G^*$ with $d(G^*)\geq 2t+1, rad(G^*)\leq r+1$ and $n(G^*)\leq r(4t^2+2t)$.
\end{theorem}
\begin{proof}
Since every graph contains a bipartite subgraph with at least half of the edges,
we may assume that $G$ is bipartite with a bipartition $(A,B)$.
Observe that if $G$ contains  a copy $L$ of $K_{t+1,2t^2+3t+1}$,
Then $L$ is a subgraph of $G$ with average degree  $2t+1$,
radius  $2\leq r+1$ 
and order at most $2t^2+4t+2<4r(t^2+t)$. So the claim holds
trivially. Hence, for the rest of the proof, we assume that $G$ is $K_{t+1,q}$-free
with $q=2t^2+3t+1$. Apply
Lemma \ref{splitting2}, with $\e=\frac{1}{r}$, $h=r$, and $b=t!(3e)^{2t}$,
and let $\beta$ be the constant $c$ returned by the lemma. By Lemma
\ref{splitting2}, there exists a partition of $V(G)$ into $L_1,\ldots, L_r$ such that 
for every $t$-matching $M$ in $G$ and for every $i,j\in [r]$ if $S$ is $i$-good then $L_j$ contains
a collection $\cC_j(M)$ of at least $t!(3e)^{2t}n^{\frac{2t+1}{r}}$ pairwise vertex disjoint $(r,i-1)$-good $t$-matchings. Furthermore, some $L_i$ contains
an $(r,r)$-good $t$-matching $M_0$. By relabeling if necessary, we may assume that $L_1$ contains $M_0$. For each vertex $x$ in $H_{t,t}(G)$, let $M(x)$ denote
the $t$-matching in $G$ that $x$ represents.

Now we define an auxiliary digraph $H$ together with a partition $U_0,U_1,\ldots, U_r$ of $V(H)$ as follows. Let $U_0$ consist of the vertex $u$ in $H_{t,t}(G)$ that corresponds to  $M_0$.  Let $U_1$ be the set of vertices in $H_{t,t}(G)$ corresponding
to $t$-matchings in $\cC_1(M_0)$. Add arcs from $u$ to all of $U_1$.
For each $i\in \{2,\ldots, r\}$, let $U_i$ be the set of vertices in $H_{t,t}(G)$ 
corresponding to  $(r,r-i)$-good $t$-matchings in $G$ that lie inside in $L_i$
(Here, we define every $t$-matching in $G$ to be $(r,0)$-good.)
For each $i\in [r-1]$ and each $x\in U_i$ we add arcs from $x$ to all the vertices in $U_{i+1}$ that represent $t$-matchings in $\cC_{i+1}(M(x))$. This defines the digraph $H$. 

By our assumptions about the $L_i$'s, for each $i\in [r-1]\cup\{0\}$, each vertex in $U_i$ has at least $t!(3e)^{2t}n^{\frac{2t+1}{r}}$
out-neighbors in $U_{i+1}$. Now, grow a breadth-first search out-tree $T$ from $u$. For each $i$, let $D_i=V(T)\cap U_i$.
For each $i\in [r-1]\cup \{0\}$, by our assumption, $D_i$ sends out at least $|D_i|t!(3e)^{2t}n^{\frac{2t+1}{r}}$ edges into $D_{i+1}$. We consider two cases.

\medskip

\noindent {\bf Case 1.} For some $i\in [r]$, $D_i$ contains a vertex $y$ that lies in the outneighborhoods of at least $t!(3e)^{2t}$ different vertices
in $D_{i-1}$.  

\medskip
Let $p=t!(3e)^{2t}$. Suppose $v$ lies in the out-neighborhoods of $x_1,\ldots, x_p\in D_{i-1}$. For each $i=1,\ldots, p$, let $A_i=V(M(x_i))\cap A$ and
$B_i=V(M(x_i))\cap B$.  Consider the list $(A_1,B_1),\ldots, (A_p,B_p)$.
The pairs in the list are not necessarily distinct. However, since $M(x_1),\ldots, M(x_p)$ are distinct matchings in $G$ and there are at most
$t!$ distinct matchings with the same bipartition,
each pair appears at most $t!$ times in the list.  So there are at least $p/t!\geq (3e)^{2t}\geq \binom{3t}{t}^2$ distinct pairs among them. Let $s=\binom{3t}{t}^2$.
Without loss of generality, suppose $(A_1,B_1),\ldots,
(A_s,B_s)$ are distinct pairs. Then either $\{A_1,\ldots, A_s\}$ or
$\{B_1,\ldots, B_s\}$ must contain at least $\binom{3t}{t}$ distinct members.
Without loss of generality, suppose $A_1,\ldots, A_{\binom{3t}{t}}$ are distinct.
By Lemma \ref{t-set-union}, there exists a collection of $2t+1$ of them, say
$A_1,\ldots, A_{2t+1}$ such that $|\bigcup_{i=1}^{2t+1} A_i|\geq 3t$.

Let $v$ denote the closest common ancestor of $x_1,\ldots, x_{2t+1}$ in $T$. Suppose $v\in D_j$.
Let $T'$ be the subtree of $T$ consisting of the directed paths from $v$ to $\{x_1,\ldots, x_{2t+1}\}$. Let $F$ be the union of $T'$ and the edges $x_1y, \cdots, x_{2t+1}y$.
Let $G^*$ be the subgraph of $G$ induced by $\bigcup_{x\in V(F)} V(M(x))$. We show that $G^*$ has average degree at least $2t+1$. For each $k=j,j+1,\ldots, i$, let $R_k=\bigcup_{x\in V(F)\cap D_k} V(M(x))$. Then $R_k\subseteq L_k$, unless $k=0$,
in which case $R_0=V(M_0)$.  Using this, one can check that 
$R_j, R_{j+1},\ldots, R_{i+1}$ are pairwise vertex disjoint in $G$. 
Also note that $R_j=V(M(v))$ and $R_{i+1}=V(M(y))$. Since $v$ is the closest common ancestor of $x_1,\ldots, x_{2t+1}$
in $T$, $v$  has at least two children in $T'$.  Let $a,b$ denote two of the children of $v$ in $T'$.
By the definition of $H$, the out-neighbors of $v$ in $H$ correspond to  a collection $\cC_{j+1}(M(v))$
of pairwise vertex disjoint $t$-matchings in $L_{j+1}$. Hence $M(a)$ and $M(b)$ are vertex disjoint. Since $G_{M(v)}$ contains $M(a)$ and $M(b)$, and
$M(a)$ and $M(b)$ are two vertex disjoint $t$-matchings,
each vertex in $R_j=V(M(v))$ has degree at least $2t$ in $G^*$. Next, let $k\in \{j+1,\ldots, i-1\}$. Let $x\in V(F)\cap D_k$. Then
$x$ has an in-neighbor $x^-$ in $D_{k-1}$ and at least one out-neighbor $x^+$ in $D_{k+1}$. Since $x^-x, xx^+\in E(H)$,
by definition, $G$ contains a copy of $H_{t,t}$ between $M(x^-)$ and $M(x)$ and a copy of $H_{t,t}$ between $M(x)$ and $M(x^+)$,
both of which are in $G^*$.  Let $w$ be any vertex in $M(x)$. Then it has $t$ neighbors in $M(x^-)$, $t$ neighbors in $M(x^+)$ and at least $1$ neighbor in $M(x)$.
Since $M(x^-), M(x), M(x^+)$ are pairwise disjoint by earlier remarks, $w$ has degree
at least $2t+1$ in $G^*$. This shows that for each $w\in R_k$, $d_{G^*}(w)\geq 2t+1$.  Finally, consider $R_i=V(M(y))$. Recall that for each $j=1,\ldots, 2t+1$,
we let $A_j=V(M_j)\cap A$ and $B_j=V(M_j)\cap B$ and by our earlier assumption,
$|\bigcup_{j=1}^{2t+1} A_j|\geq 3t$. Since $x_1y,\ldots, x_{2t+1}y\in E(H)$,
each vertex  $w$ in $M(y)\cap A$ is adjacent in $G^*$ to all of $\bigcup_{p=1}^{2t+1} B_i$. Also $w$ has at least one neighbor in $M(y)$. So $d_{G^*}(w)\geq t+1$.
Each vertex $w$ in $M(y)\cap B$ is adjacent in $G^*$ to all of $\bigcup_{p=1}^{2t+1} A_i$ and $w$ has at least one neighbor in $M(y)$. Since $\bigcup_{p=1}^{2t+1} A_i|\geq 3t$, we have $d_{G^*}(w)\geq 3t+1$. Since there are equal number of vertices in
$M(y)\cap A$ and $M(y)\cap B$, the average degree in $G^*$ among vertices in $M(y)$ is at least $2t+1$. We have earlier argued that all other vertices in $G^*$ have degree at least $2t+1$. Hence $G^*$ has average degree at least $2t+1$.
Now we have found a subgraph $G^*$ of $G$ with average degree at least $2t+1$. 
The number of vertex in $T'$ is at most
$(r-1)(2t+1)+1$ since it has $2t+1$ leaves and has height at most $r-1$. So $n(F)\leq (r-1)(2t+1)+2<r(2t+1)$ and thus $n(G^*)\leq r(2t+1)(2t)=r(4t^2+2t)$.
Also, one can check that $rad(G^*)\leq r+1$.

\medskip

\noindent {\bf Case 2.} For each $i\in [r]$ every vertex in $D_i$ lies in the out-neighborhoods of fewer than $t!(3e)^{2t}$ vertices of $D_{i-1}$.

\medskip

For each $i\in [r]$, $D_{i-1}$ sends out at least $|D_{i-1}| t!(3e)^{2t}n^{\frac{2t+1}{r}}$ edges into $D_i$ and each vertex in $D_i$ receives
fewer than $t!(3e)^{2t}$ of these edges, we have $|D_i|\geq |D_{i-1}|n^{\frac{2t+1}{r}}$. 
This yields $|D_r|\geq [n^{\frac{2t+1}{r}}]^r=n^{2t+1}>t!n^{2t}$, which is impossible since vertices in $D_r$ correspond
to distinct $t$-matchings in $G$ and there are certainly no more than $t! n^t n^t<t!n^{2t}$ distinct $t$-matchings in $G$.
\end{proof}

We can now answer Question \ref{jacques-question} for all odd $d$,
by applying Theorem \ref{odd} with $r=\fl{\frac{m}{8t^2}}$.
\begin{proposition} 
Let $t,m$ be positive integers. We have $ex(n,\cF_{2t+1,m})=O(n^{2-\frac{2}{2t+1}+\frac{8t^2}{m}})$.
\end{proposition}

%%%%%%%%%%%%%%%%%%%%%%%%%%%%%%%%%%%%%%%%%%%%%%%%%%

\section{A generalization of the cube theorem} \label{Htt-section}

In this section, we partially answered Pinchasi and Sharir's question by proving that
$ex(n,H_{t,t})=O(n^{\frac{4t}{2t+1}})$, which generalizes the cube theorem \cite{cube}
$ex(n,Q_3)=O(n^{\frac{8}{5}})$.
Given a positive integer $t$, we call the $2t$-edge tree obtained joining $t$ paths of
length $2$ at one end a {\it $t$-spider}. Eequivalently, a $t$-spider is obtained from
a $t$-edge star by subdividing each edge once. Note that a $1$-spider is just a copy
of $P_3$ or equivalently $K_{1,2}$. The proof of Lemma \ref{H1t-count} shows that
in an $n$-vertex graph $G$ with at least $Cn^{3/2}$ edges, the number of $C_4$'s
exceeds the number of $K_{1,2}$ (by any factor needed based on our choice of $C$).
There is no immediate analoguous relationship between the number of $t$-spides and the number of $H_{1,t}$'s in a general graph, mostly due to the possible
irregularities of vertex degrees in $G$. However, for dense enough $G$, one can apply a two-step regularization, introduced by Erd\H{os} and Simonovits in \cite{cube}, to obtain a nice subgraph $G'$ of $G$ on which the number of $H_{1,t}$'s exceeds the number of $t$-claws by any prescribed factor.
Given a graph, let $\lambda_t(G)$ denote the number of $t$-spiders in $G$ and
$h_{1,t}(G)$ the number of $H_{1,t}$'s in $G$. For convenience, we omit the
floors and ceilings. In the next lemma, the first part repeats Erd\H{o}s and
Simonovits' regularization process. The second part uses the regularization to
bound $\lambda_t(G')$ of the obtained subgraph $G'$.

\begin{lemma} \label{regular}
Let $t\geq 2$ be an integer. Let $C>0$ be a constant.
Let $G$ be an $n$-vertex bipartite graph with $E \geq 2^{27} (Ct!)^{\frac{1}{t+1}}n^{\frac{2t+1}{t+1}}$ edges, where $n^{1/6}>2^{11}\sqrt{2t}(\log_2 n)^4$. 
Let  $(A,B)$ be a bipartition of $G$. 
There exists a subgraph $G'$ of $G$ with a bipartition $(A',B')$ where $A'\subseteq A,
B'\subseteq B$, such that $|A'|=\frac{A}{2^i}, |B'|=\frac{B}{2^j}$ 
and  that $e(G')\geq \frac{E}{64i^2j^2}$ for some $2\leq i,j\leq 3\log n$.
Furthermore, we have  $h_{1,t}(G')\geq C \lambda_t(G')$. 
\end{lemma}
\begin{proof}
Let $r_0=0$ and for each $i\geq 1$ let $r_i=\frac{2^{i-2}}{i^2}$. 
For each $i\geq 1$, let $A_i=\{x\in A: r_{i-1} \frac{E}{|A|}\leq d_G(x) <r_i \frac{E}{|A|}\}$. Then $A=\bigcup_{i=1}^\infty A_i$. By definition,
the number of edges of $G$ that are incident to $A_1$ is less than
$\frac{E}{2}$.  So the number of edges of $G$ that are incident to $\bigcup_{i=2}^\infty A_i$ is more than $\frac{E}{2}$.
If for each $i\geq 2$ we have $|A_i|<\frac{|A|}{2^i}$, then
the number of edge of $G$ that are incident to $\bigcup_{i=2}^\infty A_i$
is less than $\sum_{i=2}^\infty \frac{2^{i-2}}{i^2}\frac{E}{|A|}\frac{|A|}{2^i}
\leq \frac{E}{4}\sum_{i=2}^\infty \frac{1}{i^2}<\frac{E}{4}\cdot 1=\frac{E}{4}$,
a contradiction. So for some $i\geq 2$, we have $|A_i|\geq \frac{|A|}{2^i}$.
Fix such an $i$. Let $A'\subseteq A_i$ be a subset
with $|A'|=\frac{|A|}{2^i}$. Let $a=|A'|$. Let $\wt{G}$ denote the subgraph of 
$G$ induced by $A'\cup B$. Let $\wt{E}$ denote the number of edges in $\wt{G}$.
By definition, $\wt{E}\geq \frac{2^{i-3}}{(i-1)^2}\frac{E}{|A|}\frac{|A|}{2^i}>
\frac{E}{8i^2}$.

For each $y\in B$, let $\wt{d}(y)$ denote the degree of $y$ in $\wt{G}$.
For each $j\geq 1$, let $B_j=\{y\in B: r_{j-1} \frac{E}{8i^2|B|}\leq \wt{d}(y)< r_j \frac{E}{8i^2 |B|}\}$.  By definition, the number of edges of $\wt{G}$ incident to $B_1$ is
less than $r_1 \frac{E}{8i^2|B|}|B|=\frac{1}{2}\frac{E}{8i^2}<\frac{\wt{E}}{2}$.
So the number of edges of $G'$ incident to $\bigcup_{i=2}^\infty B_i$ is
more than $\frac{\wt{E}}{2}$. If for each $j\geq 2$ we have $|B_j|<\frac{|B|}{2^j}$
then the number of edges of $\wt{G}$ incident to $\bigcup_{i=2}^\infty B_i$ is
less than $\sum_{i=2}^\infty \frac{2^{j-2}}{j^2}\frac{E}{8i^2|B|}\frac{|B|}{2^j}
=\frac{1}{4}\frac{E}{8i^2}\sum_{j=2}^\infty\frac{1}{j^2}<\frac{\wt{E}}{4}$, a contradiction. So there exists an $j\geq 2$ for which $|B_j|\geq \frac{|B|}{2^j}$. Fix such a $j$. Let $B'\subseteq B_j$ be a subset of $B_j$ with $|B'|=\frac{|B|}{2^j}$.
Let $G'=G[A'\cup B']$ be the subgraph of $G$ induced by $A'\cup B'$. Let
$n',E'$ denote the number of vertices and the number of edges of $G'$, respectively.
By our definition, $$E'\geq r_{j-1}\frac{E}{8i^2|B|}\cdot \frac{|B|}{2^j}
=\frac{2^{j-3}}{(j-1)^2}\frac{E}{8i^2|B|}\frac{|B|}{2^j}\geq \frac{E}{64i^2j^2}.$$
Let $\Delta_{A'}$ and $\Delta_{B'}$ denote the maximum degree in $G'$ of a vertex in $A'$ and in $B'$, respectively. By our definition of $A'$ and $B'$,
$\Delta_{A'}\leq \frac{2^{i-2}}{i^2} \frac{E}{|A|}$ and $\Delta_{B'}\leq \frac{2^{j-2}}{j^2}\frac{E}{8i^2 |B|}$. From each vertex in $A'$ there are fewer than $(\Delta_{A'})^t(\Delta_{B'})^t$ ways to grow $t$ many paths of length at most $2$ and
similarly for each vertex in $B'$. Thus we have
\begin{equation} \label{lambdat1}
\lambda_t(G')\leq (a+b)(\Delta_{A'})^t(\Delta_{B'})^t\leq \left(\frac{2^{i-2}}{i^2}\frac{E}{|A|}\frac{2^{j-2}}{j^2}\frac{E}{8i^2|B|}\right)^t (a+b)\leq 
\left(\frac{2^{i-2}2^{j-2} (64i^2j^2 E')^2}{8i^4j^2|A||B|}\right)^t (a+b).
\end{equation}
Using $|A|=2^i|A'|=2^ia$ and $|B|=2^j|B'|=2^j b$,  \eqref{lambdat1} yields
\begin{equation} \label{lambdat2}
\lambda_t(G')\leq \left( \frac{32 j^2 (E')^2}{ab}\right)^t(a+b).
\end{equation}

Next, observe that since $A_i=\{x\in A: r_{i-1}\frac{E}{|A|}\leq d(x)<r_i\frac{E}{|A|}\}$, but $\forall x\in A, d(x)\leq |B|$, we have $r_{i-1}\leq \frac{|A||B|}{E}<|A||B|\leq \frac{n^2}{4}$. That is, $\frac{2^{i-3}}{(i-1)^2}\leq \frac{n^2}{4}$. From this, one can show that $i\leq 3\log_2 n$ (using our assumption that $n$ is sufficiently large. Indeed, it suffices if  $n\geq 8(\log_2 n)^2$).
Similarly $j\leq 3\log_2 n$. Now 
$$E'\geq \frac{E}{64i^2j^2}\geq \frac{n^{\frac{2t+1}{t+1}}}{64\cdot 9 (\log_2 n)^4}\geq 
\frac{n^{5/3}}{576(\log_2 n)^4}\geq 4\sqrt{2t} n^{3/2}
\geq 4\sqrt{2t} (n')^{3/2},$$
using  $n^{1/6}>2^{11}\sqrt{2t}(\log_2 n)^4$. By Lemma \ref{H1t-count}, we have
\begin{equation} \label{H1t}
h_{1,t}(G')\geq \frac{1} {2^{5t+2} t!} \frac{(E')^{3t+1}}{a^{2t}b^{2t}}.
\end{equation}
Suppose $h_{1,t}(G')\leq C \lambda_t(G')$. Then by \eqref{lambdat2} and
\eqref{H1t}, we have
$$\frac{1}{2^{5t+2}t!} \frac{(E')^{3t+1}}{a^{2t}b^{2t}} \leq C \left(\frac{32 j^2 (E')^2}{ab}\right)^t(a+b).$$
Solving for $E'$ yields
$$(E')^{t+1}\leq C t! 2^{10t+2}j^{2t} a^tb^t(a+b).$$
Since $E'\geq \frac{E}{64i^2j^2}$, $a=\frac{|A|}{2^i}, b=\frac{|B|}{2^j}$ and
$(a+b)\leq n$, we have
$$\left(\frac{E}{64i^2j^2}\right)^{t+1}\leq Ct! 2^{10t+2}j^{2t}\frac{|A|^t}{2^{it}}\frac{|B|^t}{2^{jt}} n.$$
Hence we have
$$E^{t+1}\leq Ct! 2^{16t+8}\frac{i^{2t+2}}{2^{it}}\frac{j^{4t+2}}{2^{jt}} n^{2t+1}.$$
So, $$E\leq 2^{16}(Ct!)^{\frac{1}{t+1}} \frac{i^2}{2^{i/2}}\frac{j^4}{2^{j/2}}n^{\frac{2t+1}{t+1}}.$$

The functions $\frac{x^2}{2^{x/2}}$ and $\frac{x^4}{2^{x/2}}$ are maximize at
$x=\frac{4}{\ln 2}$ and $x=\frac{8}{\ln 2}$, respectively, which can be used to show $\frac{i^2}{2^{i/2}}<5$ and $\frac{j^4}{2^{j/2}}<328$. Since $5\cdot 328<2^{11}$, we have
$$E\leq 2^{27} (Ct!)^{\frac{1}{t+1}} n^{\frac{2t+1}{t+1}},$$
which contradicts our assumption about $E$. Therefore, we must have
 $h_{1,t}(G')\geq C \lambda_t(G')$.
\end{proof}

\begin{theorem}
Let $t\geq 2$ be a positive integer.
We have $ex(n, H_{t,t})\leq 2^{16} t n^{\frac{4t}{2t+1}}$ for sufficiently large $n$ as a function of $t$.
\end{theorem}
 \begin{proof}
Since every graph contains a bipartite subgraph of at least half of the original edges,
it suffices to consider $n$-vertex bipartite host graphs with at least
$2^{15} tn^{\frac{4t}{2t+1}}$ edges. Let $G$ be an $n$-vertex bipartite graph with
$E> 2^{15} t n^{\frac{4t}{2t+1}}$ edges. Assume that $G$ does not contain a copy of
$H_{t,t}$,  we derive a contradiction. Let $(A,B)$ be a bipartition of $G$.
Since $E>2^{15} t  n^{\frac{4t}{2t+1}}>2^{27} (8t\cdot (t-1)!)^{\frac{1}{t}} n^{\frac{2t-1}{t}}$ for large $n$,
by Lemma \ref{regular} (with $t$ replaced with $t-1$ and with $C=8t$)
there exists a subgraph $G'$ of $G$ with a bipartition $(A',B')$ where $A'\subseteq A, B'\subseteq B$, such that $|A'|=\frac{A}{2^i}, |B'|=\frac{B}{2^j}$ 
and  that $E'=e(G')\geq \frac{E}{64i^2j^2}$ for some $2\leq i,j\leq 3\log n$.
Furthermore, we have  $$h_{1,t-1}(G')\geq 8t \lambda_{t-1}(G').$$ 
In particular, note that $E'\geq \frac{E}{64 (3\log n)^4}\geq 4\sqrt{2(t-1)} n(G')^{3/2}$ for large $n$. 

For each matching $M$ in $G'$, as before, let $X'_M=N_{G'}^*(V(M)\cap B)\setminus V(M)$ and $Y'(M)=
N_{G'}^*(V(M)\cap A)\setminus V(M)$. Let $G'_M$ be the subgraph of $G'$ induced by 
$X'_M\cup Y'_M$. Let $V'_M$ and $E'_M$ denote the number of vertices and edges
in $G'_M$, respectively.  Let $h_{1,t-1}(G')$ denote the number of copies of $H_{1,t-1}$'s in $G'$. For convenience, let $a=|A'|$ and $b=|B'|$.
Since $E'>4\sqrt{2(t-1)} n(G')^{3/2}$ for large $n$, by 
Lemma \ref{H1t-count}, we have
\begin{equation} \label{h1-lower}
h_{1,t-1}(G') \geq \frac{1}{32^{t-1} (t-1)!} \frac{(E')^{3(t-1)+1}}{a^{2(t-1)}b^{2(t-1)}}.
\end{equation}

We call a $(t-1)$-matching $M$ in $G'$ {\it good} if $E'_M>4t^3 V'_M$ and
{\it bad}  if $E'_M\leq 4t^3 V'_M$.  Let $\cM$ denote the set of $(t-1)$-matchings in $G'$. Let $\cM_1$ denote the set of all good $(t-1)$-matchings in $G'$ and $\cM_2$ the set of all bad $(t-1)$-matchings in $G'$.
Note that $\sum_{M\in \cM} E'_M$ counts the total number of $H_{1,t-1}$'s in $G'$ while $\sum_{M\in \cM} V'_M$ counts the total number of $(t-1)$-spiders in $G'$. 
Since $h_{1,t-1}(G')\geq 8t^3 \lambda_{t-1}(G')$, we have $\sum_{M\in \cM} E'_M\geq 8t^3 \sum_{M\in \cM} V'_M$.
By the definition of $\cM_2$, we have

\begin{equation} \label{m2-small}
\sum_{M\in \cM_2} E'_M \leq 4t^3 \sum_{M\in \cM_2} V'_M\leq 4t^3 
\sum_{M\in \cM} V'_M\leq \frac{1}{2} \sum_{M\in \cM} E'_M.
\end{equation} 

Hence, by \eqref{h1-lower} and \eqref{m2-small}, we have

\begin{equation} \label{m1-large}
\sum_{M\in \cM_1} E'_M \geq \frac{1}{2} \sum_{M\in \cM} E'_M \geq 
\frac{1}{2^{5t+1} t!} \frac{(E')^{3t-2}}{a^{2t-2}b^{2t-2}}.
\end{equation} 

\medskip

Now, we define a $t$-matching $N$ in $G'$ to be {\it heavy} if $E'_N> 4t^2$ and {\it light} if $E'_N\leq 4t^2$.

\medskip

{\bf Claim 1.} Let $M$ be a $(t-1)$-matching in $G'$. The number of heavy $t$-matchings $N$ of $G'$ that are contained in $G'_M$ is at most $\frac{t-1}{(t-2)!}(E'_M)^{t-1}V'_M$.

\medskip

{\it Proof of Claim 1.} Suppose $M=\{a_1b_1,\ldots, a_{t-1}b_{t-1}\}$, where $a_1,\ldots, a_{t-1}\in A$ and
$b_1,\ldots, b_{t-1}\in B$. Let $N$ be any $t$-matching in $G'_M$. By definition, also 
we have $M\subseteq G'_N$. If $G'_N$ contains an edge $e$ that is vertex disjoint from $M$, then we obtain a copy of $H_{t,t}$ with parts $N$ and $M\cup e$,
contradicting $G'$ being $H_{t,t}$-free. Hence every edge in $G'_N$ must intersect
$V(M)$. Now, let $N$ be any heavy $t$-matching of $G'$ in $G'_M$. By definition, $G'_N$ has at least $4t^2$ edges. Since $V(M)$ is a vertex cover of $G'_N$, by the pigeonhole principle, some vertex in $V(M)$ lies in at least $4t^2/2(t-1)>2t$ edges of $G'_N$. We say that $N$ is {\it $w$-dense} if $w\in V(M)$ lies in at least $2t$ edges of $G'_N$.  Now, for each $i=1,\ldots, t-1$, we bound the
number of $a_i$-dense heavy $t$-matchings  and the number of $b_i$-dense heavy $t$-matchings of $G'$ in $G'_M$.  Let $L=\{u_1v_1,\ldots, u_{t-1}v_{t-1}\}$ be any $(t-1)$-matching in $G'_M$, where $u_1,\ldots, u_{t-1}\in X'_M$ and $v_1,\ldots, v_{t-1}\in Y'_M$. Let $y$ a vertex in $Y'_M$  that lies outside $L$. We show that there are
fewer than $t$ different $a_i$-dense heavy $t$-matchings in $G'_M$ that contain
$L$ and $y$. Otherwise, suppose $N_1,\ldots, N_t$ are different $a_i$-dense
heavy $t$-matchings of $G'$ that contain $L$ and $y$. For each $j=1,\ldots, t$, let
$x_jy$ denote the edge of $N_j$ that is incident to $y$. Then $x_1,\ldots, x_t\in X'_M\setminus V(L)$. 
For each $j=1,\ldots, t$, since $N_j$ is $a_i$-dense, $G'_{N_j}$ contains
a set of at least $2t$ edges that are incident to $a_i$.  We can greedily
pick distinct edges $a_ic_1, a_ic_2,\ldots, a_ic_t$ such that $c_1,\ldots, c_t\notin
\{b_1,\ldots,b_{t-1}\}$ and that $a_ic_1\in E(G'_{N_1}), a_ic_2\in E(G'_{N_2}),
\ldots, a_ic_t\in E(G'_{N_t})$. Now we claim that there is a copy of $H_{t,t}$ in $G'$.
First note that $c_1,\ldots, c_t\in N^*_{G'}(\{u_1,\ldots, u_{t-1}\})$, since for each $j$,
$a_ic_j\in E(G'_{N_j})\subseteq E(G'_L)$. Since $a_i$ is also adjacent to all of $c_1,\ldots,
c_t$, there exists a copy of $K_{t,t}$ with partite sets $\{a_i,u_1,\ldots, u_{t-1}\}$
and $\{c_1,\ldots, c_t\}$. Next, note that all of $b_1,\ldots, b_{t-1}$ are adjacent to
all of $x_1,\ldots x_t$ since $x_1y,\ldots, x_ty\in E(G'_M)$. Hence there is another copy
of $K_{t,t}$ in $G'$ with partite sets $\{x_1,\ldots, x_t\}$ and $\{b_1,\ldots, b_{t-1},y\}$. It remains to find a matching joining these two disjoint 
copies of $K_{t,t}$'s. For each $j=1,\ldots, t$, since
$a_ic_j\in E(G'_{N_j})$, we have $c_jx_j\in E(G')$. Since $L\subseteq G'_M$, we 
have $u_1b_1,\ldots, u_{t-1}b_{t-1}\in E(G')$. Since $y\in V(G'_M)$, we have
$a_iy\in E(G')$. Hence, we obtain a copy of $H_{t,t}$ in $G'$, a contradiction.
Hence, for each $(t-1)$-matching $L$ in $G'_M$ and each $y\in Y'_M \setminus V(L)$,
there are at most $(t-1)$ many $a_i$-dense heavy $t$-matchings of $G'$ containing $L$ and $y$.
Hence the number of $a_i$-dense heavy $t$-matchings in $G'_M$ is at most 
$(t-1)\frac{(E'_M)^{t-1}}{(t-1)!}|Y'_M|$.  By a similar argument, the number of $b_i$-dense heavy $t$-matchings in $G'_M$ is at most $(t-1)\frac{(E')_M^{t-1}}{(t-1)!}|X'_M|$.  Therefore, the total number
of heavy $t$-matchings of $'G$ that lie in $G'_M$ is at most $(t-1)^2\frac{(E'_M)^{t-1}}{(t-1)!}(|X'_M|+|Y'_M|)<\frac{t-1}{(t-2)!}(E'_M)^{t-1}V'_M$. \qed

\medskip

{\bf Claim 2.} For each $M\in \cM_1$, the number of light $t$-matchings of $G'$  in $G'_M$ is at least $\frac{1}{4} \frac{(E'_M)^t}{t!}$.
 
\medskip

{\it Proof.} Let $M\in \cM_1$. By definition, $E'_M> 8t V'_M$. Obviously $\Delta(G'_M)\leq V'_M$. Since $E'_M>4t^3\Delta(G'_M)$, by Lemma \ref{matching-count}, the number of $t$-matchings in $G'_M$ is at least $\mu'=\frac{1}{2}\frac{(E'_M)^t}{t!}$. By Claim 1, among them the number of heavy $t$-matchings is at most $\mu''=\frac{(t-1)}{(t-2)!} (E'_M)^{t-1}V'_M$. Since $E'_M\geq 4t^3 V'_M$, one can check that $\mu''<\frac{1}{2}\mu'$.
Hence, the number of light $t$-matchings of $G'$ in $G'_M$ is at least 
$\frac{1}{2} \mu'=\frac{1}{4} \frac{(E'_M)^t}{t!}$. \qed

\medskip

Let $W$ denote the number of pairs $(M,N)$ where $M\in \cM_1$ and $N$ is
a light $t$-matching of $G'$ that lies in $G'_M$. By Claim 2, \eqref{m1-large}, and convexity, we have 
\begin{equation} \label{S-lower}
W\geq \sum_{M\in \cM_1} \frac{1}{4}\frac{(E'_M)^t}{t!}
\geq \frac{1}{4t!} \frac{ (\sum_{M\in \cM_1} E'_M)^t}{|\cM_1|^{t-1}}
\geq \frac{1}{4t!} \frac{ (\sum_{M\in \cM_1} E'_M)^t}{((E')^{t-1})^{t-1}}
=\frac{1}{2^{5t^2+t+2} (t!)^{t+1}} \frac{(E')^{2t^2-1}}{a^{2t^2-2t}
b^{2t^2-2t}}.
\end{equation}

On the other hand, for each light $t$-matching $N$, by definition $E'_N\leq 4t^2$.
So certainly there are at most $(4t^2)^{t-1}<4^t t^{2t}$ many $(t-1)$-matchings
$M$ in $\cM_1$ that lie in $G'_N$. Equivalently, $N$ lies in $G'_M$ for fewer than
$4^t t^{2t}$ members of $\cM_1$. Hence, 
\begin{equation} \label{S-upper}
W\leq 4^t t^{2t} (E')^t.
\end{equation}

By \eqref{S-lower} and \eqref{S-upper}, we have
$$\frac{1}{2^{5t^2+t+2} (t!)^{t+1}} \frac{(E')^{2t^2-1}}{a^{2t^2-2t}b^{2t^2-2t}}
\leq 4^t t^{2t} (E')^t.$$

Solving for $E'$ and relaxing the inequalities along the way, we get
$$(E')^{2t^2-t-1}\leq 2^{5t^2+3t+2} t^{t^2+3t}a^{2t^2-2t} b^{2t^2-2t}.$$
\begin{equation}
E'\leq \left (2^{5t^2+3t+2} t^{t^2+3t}\right)^\frac{1}{2t^2-t-1} a^{\frac{2t}{2t+1}}b^{\frac{2t}{2t+1}} < 128 t a^{\frac{2t}{2t+1}}b^{\frac{2t}{2t+1}} . 
\end{equation}
Since $E'\geq \frac{E}{64i^2j^2}$, $a=\frac{|A|}{2^i}$, $b=\frac{|B|}{2^j}$, 
we get
$$\frac{E}{64i^2 j^2} <128t \left (\frac{|A|}{2^i}\right)^{\frac{2t}{2t+1}}\left (\frac{|B|}{2^j}\right)^{\frac{2t}{2t+1}}<\frac{128t}{2^{\frac{4i}{5}}2^{\frac{4j}{5}} } n^{\frac{4t}{2t+1}}.$$
Solving for $E$ and using $\frac{i^2}{2^{4i/5}}<2$ for all $i$, as in the proof of Lemma \ref{regular}, we get
\begin{equation}
E< 2^{13} t \left(\frac{i^2}{2^{4i/5}}\right)^2 n^{\frac{4t}{2t+1}} < 2^{15}t n^{\frac{4t}{2t+1}} .
\end{equation}

This contradicts our assumption about $E$ and completes the proof.
\end{proof}

%%%%%%%%%%%%%%%%%%%%%%%%%%%%%%%%%%%%%%%%%

\section{Concluding remarks}

Using supersaturation of the even cycle $C_{2k}$ for $n$-vertex graphs with $\Omega(n^{1+\frac{1}{k}+\epsilon})$ edges, we can also give an affirmative
answer to Question \ref{jacques-question}, for average degree $d$ of the form
$d=2+\frac{2}{p}$, for any integer $p\geq 2$. However, Question \ref{jacques-question}
is generally open for other rational numbers $d$.
Perhaps a question that is more interesting is to explore the analogous problem for 
regular subgraphs of bounded order.
There is a line of well-known prior work on the existence of regular subgraphs
in ``dense'' host graphs.
Answering a question of Erd\H{os} and Sauer \cite{Erdos}, Pyber \cite{pyber} proved
that every $n$-vertex graph with at least $32k^2 n\ln n$ edges contains a $k$-regular subgraph. On the other hand, Pyber, R\"odl, and Szemer\'edi \cite{PRS} established the
existence of $n$-vertex bipartite graphs with $c n\ln\ln n$ edges that do not contain any regular subgraphs. It'll be interesting to explore the edge-density needed to force
regular subgraphs of bounded order.  

\begin{problem}
{\rm For all integers $m,d\geq 3$, let $\cR_{d,m}$ denote the family of $d$-regular
graphs on at most $m$ vertices. Find good estimates on $ex(n,\cR_{d,m})$.}
\end{problem}

An interesting family of $d$-regular graphs when $d=2t$ is even is the {\it $t$-blowup} of a cycle, where the $t$-blowup of a graph is obtained by replacing each vertex with an independent set of $t$ vertices and replacing each edge with the corresponding $K_{t,t}$. Let $\cC_t$ denote the family of all $t$-blowups of cycles. We pose the following question on $\cC_t$.

\begin{question}
Is it true that for any $\epsilon>0$, $ex(n,\cC_t)=O(n^{2-\frac{1}{t}+\epsilon})$?
\end{question}

Finally, it will be interesting to answer the question of Pinchasi and Sharir \cite{PS}
on whether $ex(n,H_{s,t})=O(n^{\frac{4s}{2s+1}})$ when $s<t$.
%%%%%%%%%%%%%%%%%%%%%%%%%%%%%%%%%%%%%%%%%

\section{Acknowledgment}
The first author is much indebted to  Professor Vojt\v ech R\"odl for 
stimulating discussions and for suggesting useful approaches on the problem
during his visit of Emory University.

%%%%%%%%%%%%%%%%%%%%%%%%%%%%%%%%%%%%%%%%%%%%%%%%%%


\begin{thebibliography}{99}

\bibitem{AS} N. Alon, J. Spencer, The Probabilistic Method, 3rd edition, Wiley-Interscience, 2008.

\bibitem{bukh} B. Bukh: Set families with a forbidden subposet, \emph{Electronic J. Combin.}
\textbf{16} (2009), \#R142.

\bibitem{Boehnlein-Jiang} E. Boehnlein, T. Jiang: Set systems with a forbidden induced subposet, \emph{Combin. Probab. Comput.} \textbf{21}, (2012), 496-511.

\bibitem{CFS} D. Conlon, J. Fox, B. Sudakov: An approximate version of Sidorenko's conjecture, \emph{Geom. Funct. Anal.} \textbf{20} (2010), 1354-1366.

\bibitem{DLS} S. Das, C. Lee, B. Sudakov: Rainbow Tur\'an problem for even cycles, 
\emph{European J. Combinatorics} \textbf{34}, (2013), 905-915.

\bibitem{Erdos} P. Erd\H{o}s: On the combinatorial problems that I would most 
like to see solved, \emph{Combinatorica} \textbf{1} (1981), 25-42.

\bibitem{NV} A. Naor, J. Verstra\"ete: Parity check matrics and product representations of squares, \emph{Combinatorica} \textbf{28} (2) (2008), 163-185.

\bibitem{cube} P. Erd\H{o}s, M. Simonovits:  Some extremal problems in graph theory,
\emph{Combinatorial Theory and Its Applications 1} (Proc. Colloq. Balatonf\"ured, 1969),
North Holland, Amsterdam, 1970, 370-390.

\bibitem{FS} Z. F\"uredi, M. Simonovits: The history of the degenerate (bipartite) extremal graph problems, arXiv:1306.5167.

\bibitem{JLR} S. Janson, T. Luczak, A. Rucinski, Random Graphs, John Wiley \& Sons, Inc, 2000.


\bibitem{PS} R. Pinchasi, M. Sharir: On graphs that do not contain the cube and related problems, \emph{Combinatorica } \textbf{25} (5) (2005), 615-623.

\bibitem{pyber} L. Pyber: Regular subgraphs of dense graphs, \emph{Combinatorica} \textbf{5} (1985), 347-349.

\bibitem{PRS} L. Pyber, V. R\"odl, E. Szemer\'edi: Graphs without $3$-regular subgraphs,
\emph{J. Combinatorial Theory Ser. B} \textbf{63} (1995), 41-54.

\bibitem{simonovits} M. Simonovits: Problem collection from the Institut Mittag-Leffler 
programme ``Graphs, hypergraphs, and computing", Problem 9.1.

\bibitem{verstraete} J. Verstra\"ete: personal communications.

\bibitem{west} D.B. West: Introduction to graph theory, second edition, Prentice Hall, 2001.

\end{thebibliography}
\end{document}